\documentclass[12pt]{article}

\usepackage{mathtools}
\usepackage{algorithm}
\usepackage{algorithmic}
\usepackage{mathtools}
\usepackage{amsthm}
\usepackage{stmaryrd}
\usepackage[margin=1.05in]{geometry}
\usepackage{titling}
\usepackage{xspace}
\usepackage{blkarray}
\usepackage{multicol}
\usepackage{bm}
\usepackage{url}
\usepackage{tikz}
\usepackage{pgfplots}
\usepackage{tikz-3dplot}
\usepackage{mathrsfs}
\usepackage{tikz-cd}
\usepackage{multirow}
\usepackage{comment}
 \allowdisplaybreaks
 
\definecolor{mycolor1}{rgb}{1,0,0}
\definecolor{mycolor2}{rgb}{0,1,0}
\definecolor{mycolor3}{rgb}{0,0,1}

\usepackage{multicol}

\usepackage[font=footnotesize,labelfont=bf]{caption}
\usepackage{subfig} 

\usetikzlibrary{external}
\usetikzlibrary{patterns}
\pgfplotsset{
  log x ticks with fixed point/.style={
      xticklabel={
        \pgfkeys{/pgf/fpu=true}
        \pgfmathparse{exp(\tick)}%
        \pgfmathprintnumber[fixed relative, precision=3]{\pgfmathresult}
        \pgfkeys{/pgf/fpu=false}
      }
  }}
\usepackage{array}
\usepackage{enumitem}
\usepackage{authblk}
\usepackage{blindtext}
\usepackage{hyperref}
\hypersetup{
    colorlinks=true,
    linkcolor=orange!80!black,
    filecolor=magenta,      
    urlcolor=cyan,
    citecolor=blue!50!white,
  }
\usepackage{adjustbox}
\newcolumntype{R}[2]{%
    >{\adjustbox{angle=#1,lap=\width-(#2)}\bgroup}%
    l%
    <{\egroup}%
}
\makeatother
\usepackage{makecell}
\usepackage{amssymb}
\usepackage{subfig}
\usepackage{todonotes}
\makeatletter
\newcommand{\xdashrightarrow}[2][]{\ext@arrow 0359\rightarrowfill@@{#1}{#2}}
\makeatother
\let\emptyset\varnothing

\usepackage{cleveref}

\DeclareMathOperator{\lcm}{lcm}

\DeclareMathOperator{\Var}{Var}
\theoremstyle{plain}
\newtheorem*{theorem*}{Theorem}
\newtheorem{theorem}{Theorem}[section]
\newtheorem{lemma}{Lemma}[section]

\newtheorem{remark}{Remark}[section]
\newtheorem{proposition}{Proposition}[section]
\newtheorem{corollary}{Corollary}[section]

\theoremstyle{definition}
\numberwithin{equation}{section}
\newtheorem{definition}{Definition}[section]

\newtheorem{examplex}{Example}
\newenvironment{example}
  {\pushQED{\qed}\examplex}
  {\popQED\endexamplex}

\newcommand{\N}{\mathbb{N}}

\newcommand{\Pic}{\textup{Pic}}

\newcommand{\Reg}{\textup{Reg}}
\newcommand{\id}{\textup{id}}

\newcommand{\relInt}{\textup{Relint}}

\newcommand{\im}{\textup{im}}
\newcommand{\m}{\mathfrak{m}}

\newcommand{\B}{\mathcal{B}}

\newcommand{\Res}{\textup{Res}}

\newcommand{\Hom}{\textup{Hom}}

\newcommand{\I}{\mathscr{I}}

\newcommand{\Div}{\textup{Div}}

\newcommand{\codim}{\textup{codim}}
\newcommand{\Spec}{\textup{Spec}}
\newcommand{\Cl}{\textup{Cl}}

\newcommand{\eval}{\textup{ev}}
\newcommand{\ord}{\textup{ord}}

\newcommand{\A}{\mathscr{J}}
\renewcommand{\AA}{\mathscr{A}}

\newcommand{\OO}{\mathscr{O}}
\newcommand{\QQ}{\mathscr{Q}}

\newcommand{\Z}{\mathbb{Z}}

\newcommand{\C}{\mathbb{C}}
\newcommand{\Diff}{\mathscr{D}}
\newcommand{\f}{\hat{f}}
\newcommand{\Q}{\mathbb{Q}}
\newcommand{\PP}{\mathbb{P}}

\newcommand{\V}{V}

\newcommand{\1}{\mathbf{1}}
\newcommand{\z}{\zeta}

\newcommand{\D}{\delta}
\newcommand{\DD}{\delta^+}
\renewcommand{\z}{\zeta}

\newcommand{\R}{\mathbb{R}}

\newcommand{\OX}{{\mathscr{O}_{X}}}

\newcommand{\OY}{{\mathscr{O}_{Y}}}
\newcommand{\Ko}{\mathcal{K}}

\newcommand{\Span}{\textup{span}}
\newcommand{\MV}{\textup{MV}}
\providecommand{\keywords}[1]{\small \textbf{Key words ---} #1}
\providecommand{\classification}[1]{\small \textbf{AMS subject classifications ---} #1}
\newcommand{\pair}[1]{\langle{#1}\rangle}
\newcommand{\ideal}[1]{\left \langle {#1}  \right \rangle}
\newcommand{\Dim}{\textup{dim}}

\newcommand{\Picbpf}{\Pic^{\circ}}
\newcommand{\QPic}{\mathbb{Q}\textup{Pic}}

\newcommand{\Ubold}{\mathbf{U}}

\usepackage{mathptmx}

\begin{document}
\title{Toric Eigenvalue Methods for \\ Solving Sparse Polynomial Systems}
\author[ ]{Mat{\'\i}as R. Bender\thanks{Department of Mathematics, Technische Universit\"{a}t Berlin, \texttt{mbender@math.tu-berlin.de}} ~ Simon Telen\thanks{Max Planck Institute for Mathematics in the Sciences, Leipzig, \texttt{simon.telen@mis.mpg.de}}}
\date{}
\maketitle

\vspace{-3\baselineskip}
\begin{abstract}
  We consider the problem of computing homogeneous coordinates of
  points in a zero-dimensional subscheme of a compact, complex toric
  variety $X$. Our starting point is a homogeneous ideal $I$ in the
  Cox ring of $X$, which in practice might arise from homogenizing a
  sparse polynomial system.
  We prove a new eigenvalue theorem in the toric compact setting,
  which leads to a novel, robust numerical approach for solving this
  problem.
  Our method works in particular for systems having isolated solutions
  with arbitrary multiplicities.
  It depends on the multigraded regularity properties of $I$. We study
  these properties and provide bounds on the size of the matrices in our approach when $I$ is a complete
  intersection.
\end{abstract}

\keywords{\small solving polynomial systems, sparse polynomial systems, toric varieties, Cox rings, eigenvalue theorem, symbolic-numeric algorithm}

\classification{\small 14M25, 
65H04, 
65H10 
}

\vspace{-.5\baselineskip}

\section{Introduction}
\label{sec:introduction} The problem of solving a system of polynomial
equations is ubiquitous in both pure and applied mathematics and in
several engineering disciplines. Here we will consider only the important case where the solution set is finite. In many applications, the coefficients of the equations come from (noisy) measurements and
extremely short computation times are required. Moreover, it is often
sufficient to have numerical approximations of the solutions to the
system. This establishes a need for the development of robust,
numerical algorithms for solving these problems in floating point
arithmetic. Existing numerical methods include \emph{homotopy
  continuation methods}, which solve the problem using continuous
deformation techniques, and \emph{algebraic methods}, which solve the
system by performing numerical linear algebra computations. For an
overview of these techniques and applications, see
\cite{dickenstein_solving_2005,sommese_numerical_2005,sturmfels2002solving,elkadi2007introduction,cox2020applications}
and references therein.
In this paper, our goal is to present a \emph{robust}, yet \emph{efficient} numerical
algebraic method for solving polynomial equations and to develop the necessary theory. We now explain this in more detail. 

When using numerical algorithms, often the best one can hope for is to find the exact solution to a problem with slightly perturbed input data. In our context, these input data are the coefficients of the polynomials defining the system. One of the intrinsic challenges in solving polynomial equations numerically comes from the fact that small perturbations of these coefficients 
may change the geometry of the solution set significantly. For instance, this perturbation can introduce new solutions ``near
infinity''. Moreover, it typically causes solutions with higher multiplicity to split up into several distinct solutions. To make matters worse, when there are more equations than unknowns, perturbing coefficients leads to a system with no
solutions at all.
By a \emph{robust} numerical algorithm we mean one that can approximate the solutions of the original system, even in the presence of solutions at/near infinity or with higher multiplicities. Our strategy is to employ a \emph{toric compactification} of the solution space to deal with solutions at infinity, and generalize previously developed methods for dealing with multiplicities in this setting. Our toric compactification takes the \emph{polyhedral structure} of the equations into account. This has the usual beneficial effect on the efficiency of our numerical solver, as we clarify below.


 \paragraph{Previous work.}
 
The method in this paper belongs to the class of numerical eigenvalue methods for root finding, see e.g. \cite{emiris_complexity_1996,stetter2004numerical,massri2016solving,telen2019numerical}. 
Such methods rely on different variants of the \emph{eigenvalue theorem}.
Classically, this theorem states the following. Consider polynomials
$f_1,\dots,f_s \in R := \C[x_1,\dots,x_n]$ generating an ideal $I := \ideal{f_1,\ldots, f_s}$, such that $f_1 = \cdots = f_s = 0$ has finitely many solutions in $\C^n$. We denote the set of solutions by $V(I) \subset \C^n$. The eigenvalues of the \emph{multiplication map} 
$m_g : R/I \rightarrow
R/I$,
$h \mapsto m_g(h) := g \, h$, are given by $g(z), z \in V(I)$. 
Moreover, if we have a vector space basis of
$R/I$ given by $b_1,\dots,b_\delta \in R/I$ and $I$ is radical,
then the left eigenvectors of $m_g$ are the row vectors $(b_1(z), \ldots, b_\delta(z))$, for each solution $z \in V(I)$. The coordinates of the solutions can be recovered from these vectors.
We review these classical results in more detail in
\Cref{subsec:affinefat} and we refer to \cite{cox_stickelberger_2020}
for a recent historical overview on this theorem.

The choice of basis $b_i$ of
the quotient ring influences the accuracy with which we can approximate the multiplication maps $m_g$ \cite{telen2018stabilized}.
As we illustrate in \Cref{ex:intro}, classical affine methods to compute these maps, such as the Canny-Emiris sparse resultant matrix
\cite{canny_efficient_1993,emiris_complexity_1996}, are not robust in the sense explained above. They might require the inversion of a near-singular matrix, resulting in large rounding errors. These methods consider an a priori fixed basis for the quotient
ring, i.e. given by mixed cells \cite[Sec.~5]{emiris_complexity_1996},
which is often not the best choice for the specific system at hand, see the discussion in
\cite[Sect.~7]{telen2018stabilized}.

Methods exploiting the sparse, polyhedral structure of the equations often lead to much more efficient algorithms than the classical, `total degree based' approaches. In the context of homotopy solvers, this explains the great success of polyhedral homotopies \cite{verschelde1994homotopies,huber1995polyhedral} as an alternative for the more classical total degree homotopies, see e.g. \cite[Ch.~15]{burgisser_condition_2013}. In algebraic methods, exploiting these structures leads to smaller matrix constructions. Examples include the above mentioned sparse resultant matrices as opposed to the classical Macaulay resultant matrix \cite{emiris_complexity_1996,emiris1999matrices,massri2016solving}, matrices in sparse Gr\"obner basis algorithms \cite{bender_towards_2018,bender2019gr} and matrices representing truncated normal forms \cite{telen2018solving}. 

Whereas polyhedral methods are used to increase the \emph{efficiency}
of both symbolic and numerical algorithms, a strategy to improve the
\emph{robustness}
of numerical solvers is to \emph{compactify the solution space}. Homotopy path tracking in (multi-)projective spaces has the advantage that there are no diverging paths \cite[Sec.~5.1]{bates2013numerically} and ill-conditioned matrices in normal form methods can be avoided by using homogeneous interpretations \cite[Sec. 5 \& 6]{telen2018solving}. 
The practical approach to compactification is to homogenize the input equations to a graded ring, naturally associated to the considered compact space. For sparse systems, the standard (multi-)homogeneous compactifications typically lead to a homogeneous system of equations defining spurious positive dimensional solution components on the boundary (or \emph{at infinity}), which are often highly singular. This causes trouble for both homotopy continuation and algebraic algorithms. 
In this paper, we use toric compactifications for which in most cases, meaning for \emph{general} choices of the coefficients, no spurious components are introduced on the boundary. The (multi-)homogeneous compactifications can be considered as special cases of the ones we construct.

The first steps towards a numerically robust algorithm for sparse
polynomial systems were taken by the second author in
\cite{telen2019numerical}. To solve these systems, homogeneous
polynomials in the \emph{Cox ring} $S$ of a compact toric variety $X$ were considered. The
algorithm computes homogeneous coordinates of the solutions from the
eigenvalues of a \emph{multiplication map} in certain (multi-)degrees
$\alpha$ of $S$.
However, it is required that the solutions of the system have
multiplicity one and they belong to the simplicial part of
$X$. Moreover, the actual value of $\alpha$ necessary for this
approach to work was not determined.
We point out that the recent Cox homotopy algorithm
\cite{duff2020polyhedral} uses homotopy continuation to solve systems
in the Cox ring.

\paragraph{Contributions.}
In this work, we use eigenvalue computations to solve polynomial systems on the toric variety $X$. The equations are assumed to have isolated, possibly singular solutions, which need not belong to the simplicial part of $X$.
We introduce and study a new notion of \emph{regularity} for these
zero-dimensional systems, similar to the classical Castelnuovo-Mumford
regularity for projective space. As opposed to previous definitions
that consider degrees in the \emph{Picard group} $\Pic(X)$,
e.g. \cite{maclagan2003multigraded}, we allow degrees in the more
general \emph{class group} $\Cl(X)$. This way we can work with smaller matrices in
our eigenvalue algorithm.

\begin{definition}  \label{def:regintro}
Consider a homogeneous ideal $I \subset S$ defining a
zero-dimensional subscheme $V_X(I) \subset X$ of degree
$\DD$. Let $B \subset S$ be the irrelevant ideal of the Cox ring of $X$. The \emph{regularity} of $I$ is 
$$ \Reg(I) =  \{ \alpha \in \Cl(X) ~|~ \dim_\C (S/I)_\alpha = \DD, I_\alpha = (I:B^\infty)_\alpha, \textup{no point in $V_X(I)$ is a basepoint of $S_\alpha$} \}.$$
A tuple $(\alpha, \alpha_0) \in \Cl(X)^2$ is called a \emph{regularity pair} if $\alpha, \alpha + \alpha_0 \in \Reg(I)$ and no point in $V_X(I)$ is a basepoint of $S_{\alpha_0}$.
\end{definition}
In this definition, a point $p \in X$ is called a \emph{basepoint} of
$S_\alpha$ if all $f \in S_\alpha$ vanish at $p$. For a degree $\alpha \in \Reg(I)$, the graded piece $(S/I)_\alpha$ carries all the geometric information of the zero-dimensional subscheme $V_X(I)$ defined by $I$. We formalize this (\Cref{thm:regImpliesIso}) and discuss other definitions of regularity in Subsection
\ref{subsec:defregularity}.
Our first main result generalizes the classical affine eigenvalue
theorem to the toric setting.
\begin{theorem*}[Toric eigenvalue theorem; {Theorem \ref{thm:toriceval}}]
Let $I \subset S$ be a homogeneous ideal such that $V_X(I)$ is zero-dimensional of degree $\DD$. Let $\z_1, \ldots, \z_\D$ denote the points in $V_X(I)$ and let $\mu_i$ be the multiplicity of $\z_i$ (so that $\DD = \mu_1 + \cdots + \mu_\D$). Let $(\alpha, \alpha_0)$ be a regularity pair. For $g, h_0 \in S_{\alpha_0}$ such that the rational function $\phi = g/h_0$ is regular at $V_X(I)$, the map $M_{\phi}: (S/I)_{\alpha+\alpha_0} \rightarrow (S/I)_{\alpha+\alpha_0}$ representing `multiplication with $\phi$' has eigenvalues $\phi(\z_i), i = 1, \ldots, \D$, where $\phi(\z_i)$ has algebraic multiplicity $\mu_i$. 
\end{theorem*}

Our proof presents an explicit description of the eigenstructure of
the multiplication maps in terms of differential operators defining
the multiplicity structure, adapting known ideas from the affine case
\cite{marinari1993grobner,moller1995multivariate}.
In \Cref{subsec:algorithm}, we present an algorithm based on the toric
  eigenvalue theorem to solve sparse systems.
This algorithm uses tools from numerical linear algebra, such as QR
with column pivoting and SVD, to improve the numerical stability of
the computed multiplication maps. As a result, it computes bases for
the quotient ring which might differ from the classical mixed cells as
in \cite{pedersen_mixed_1996,emiris_complexity_1996}.

Our algorithm assumes that a regularity pair $(\alpha, \alpha_0)$ is
provided, which is our main motivation for studying the regularity of
zero-dimensional ideals. We provide a general criterion for extending
a degree $\alpha \in \Reg(I)$ to a regularity pair via evaluation of
the Hilbert function (Theorem \ref{thm:regQPic}). While the problem of
explicitly describing regularity pairs of general zero-dimensional
ideals seems out of reach at this moment, our second main result
provides a conclusive answer for ideals coming from \emph{square}
systems. Such ideals can be generated by $n = \dim X$ polynomials. They correspond to the most ubiquitous case in
practice. Geometrically, they define \emph{complete intersections} on
$X$.

\begin{theorem*}[Regularity for complete intersections; {Theorem~\ref{thm:regQCartier}}] Let
  $I = \ideal{f_1, \ldots, f_n} \subset S$ with $f_i \in S_{\alpha_i}$
  such that $\alpha_i \in \Pic(X)$ is basepoint free and $V_X(I)$ is
  zero-dimensional.  For any nef $\alpha_0 \in \QPic(X)$ such that
  there is $h \in S_{\alpha_0}$ which does not vanish at any point of
  $V_X(I)$, the tuple $(\alpha, \alpha_0)$ is a regularity pair, with
  $\alpha = \alpha_1 + \cdots + \alpha_n$. In particular,
  $(\alpha, \alpha_0)$ is a regularity pair for any basepoint free
  $\alpha_0 \in \Pic(X)$.
\end{theorem*}

The dimension of the $\C$-vector space $S_{\alpha+\alpha_0}$, where
$(\alpha,\alpha_0)$ is a regularity pair, determines the size of the matrices involved in our eigenvalue
algorithm, as described in \Cref{subsec:algorithm}.
Therefore, we are interested in finding regularity pairs for which $\dim_\C S_{\alpha+\alpha_0}$ is as
small as possible.
By the previous theorem, in the case where $I$ can be generated by
$n = \dim X$ elements, this dimension is bounded by
$\dim_\C S_{\alpha_0 + \alpha_1 + \cdots + \alpha_n}$ for any
basepoint free $\alpha_0 \in \Pic(X)$.
In practice, when our system comes from the homogenization of a sparse
polynomial system, this bound correspond to a matrix of size
proportional to the Minkowski sum of the input polytopes, which,
roughly speaking, agrees with the size of the matrices obtained by
other methods based on the Canny-Emiris sparse resultant matrix as
\cite{emiris_complexity_1996,massri2016solving}.
However, this bound may be pessimistic, leading to unnecessarily big
matrices. For this reason, in Section \ref{subsec:generalizations} we
study sparse polynomial systems with some extra structure and obtain
regularity pairs related to smaller matrices. These systems include
unmixed, classical homogeneous, weighted homogeneous and
multihomogeneous square systems.

We conclude with an example that illustrates the effectiveness of our approach by comparing it to a classical sparse-resultant-based technique for an input system with a solution near
infinity.

\begin{example} \label{ex:intro}
Consider the system of equations on $(\C^*)^2$ given by  $\f_1 = \f_2 = 0$, with
$$
\left\{
\begin{array}{l}
   \f_1 :=    -1+t_1+t_1^2+t_2+t_1 \, t_2, \\
   \f_2 :=  -2+2 \, t_1 + (5-2\varepsilon) \, t_1^2 + 4 \, t_2 + 5 \, t_1 \, t_2.
\end{array}
\right.
$$
The system involves a parameter $\varepsilon$, for which we will consider the real values $\varepsilon \in [0,1]$. For $\varepsilon \in (0,1]$, the system has 3 solutions in $(\C^*)^2$. As $\varepsilon \rightarrow 0$, one out of the three solutions moves towards infinity. The norm of the coordinate vector in $(\C^*)^2$ of the largest solution is plotted in the left part of Figure \ref{fig:intro}. As we will see in Example \ref{ex:counter}, this diverging solution is moving to a torus invariant divisor on a Hirzebruch surface. 

The solutions for $\varepsilon \in (0,1]$ can be computed via the eigenvectors of the Schur complement of a Canny-Emiris sparse resultant matrix, as in \cite[Sec.~3.5.1]{dickenstein_solving_2005}. We have done this for $\varepsilon = 10^{-e}$, where $e = 0, 1/2, 1, 3/2, \ldots, 14$. Computing the Schur complement requires the inversion of a square submatrix of the resultant matrix. We have plotted its condition number along with the solution norm on the left part of Figure \ref{fig:intro}. The drastic growth of this condition number causes big rounding errors on the coordinates of \emph{all} solutions, not only the one that drifts off to infinity. This is shown in the right part of Figure \ref{fig:intro}, where we plot the \emph{residual} of the numerically obtained approximate solutions. This is a measure for the relative backward error, see \cite[Appendix C]{telen2020thesis} for details. The same figure also shows the residuals for the approximate solutions obtained via the method presented in this paper. The results clearly illustrate that our method can deal perfectly with solutions drifting off to the boundary of the torus. We point out that, when solving this system using the state-of-the-art Julia homotopy package \texttt{HomotopyContinuation.jl} (v2.5.7) \cite{breiding2018homotopycontinuation}, only two solutions are computed for $e \geq 9$. The path leading to the largest solution is truncated prematurely, as it is assumed to diverge. This issue is addressed by a toric compactification for homotopy methods in \cite{duff2020polyhedral}.
\begin{figure}
\begin{tikzpicture}
\begin{axis}[%
width=2.5in,
height=1.5in,
at={(0.772in,0.516in)},
scale only axis,
xmin=0,
xmax=14,
xlabel style={font=\color{white!15!black}},
xlabel={$e$},
ymode=log,
ymin=1e-1,
ymax=1e16,
yminorticks=true,
axis background/.style={fill=white}
]
\addplot [color=black, mark size=1.0pt, mark=*, mark options={solid, black}, forget plot]
table[row sep=crcr]{
0.0 4.1815405503520555\\ 
0.5 5.957150464643101\\ 
1.0 16.001032698917825\\ 
1.5 46.748764961047335\\ 
2.0 143.5115747995016\\ 
2.5 449.3249269911578\\ 
3.0 1416.3317081366538\\ 
3.5 4474.2562698316015\\ 
4.0 14144.256625840642\\ 
4.5 44723.48076921423\\ 
5.0 141423.4775102386\\ 
5.5 447215.71670960804\\ 
6.0 1.4142156822330581e6\\ 
6.5 4.472138056019696e6\\ 
7.0 1.4142137720240906e7\\ 
7.5 4.472136131277173e7\\ 
8.0 1.414213532799484e8\\ 
8.5 4.4721378305080706e8\\ 
9.0 1.414213949038181e9\\ 
9.5 4.472122155796635e9\\ 
10.0 1.4141979657219147e10\\ 
10.5 4.472123609274677e10\\ 
11.0 1.4141339871251678e11\\ 
11.5 4.4700374571891675e11\\ 
12.0 1.4135310526750652e12\\ 
12.5 4.459329210982372e12\\ 
13.0 1.3846735570553816e13\\ 
13.5 4.3889815477184375e13\\ 
14.0 1.403418168180207e14\\ 
};  \label{fignorms} 
\addplot [color=mycolor3, mark size=1.0pt, mark=*, mark options={solid, mycolor3}, forget plot]
        table[row sep=crcr]{
0.0 159.6989634263153\\ 
0.5 870.822490232427\\ 
1.0 1197.4262562527238\\ 
1.5 3546.472820900721\\ 
2.0 11152.112180808921\\ 
2.5 35247.579924669786\\ 
3.0 111457.16889062131\\ 
3.5 352457.07650581654\\ 
4.0 1.1145669647088167e6\\ 
4.5 3.524570437569629e6\\ 
5.0 1.1145670706867192e7\\ 
5.5 3.524570586835758e7\\ 
6.0 1.1145670871219258e8\\ 
6.5 3.5245705991858333e8\\ 
7.0 1.1145670922150443e9\\ 
7.5 3.5245705576951795e9\\ 
8.0 1.1145671016084583e10\\ 
8.5 3.524570792617162e10\\ 
9.0 1.1145672878106859e11\\ 
9.5 3.5245647639252563e11\\ 
10.0 1.1145717722219033e12\\ 
10.5 3.524544172630287e12\\ 
11.0 1.1145295517255174e13\\ 
11.5 3.524461791458049e13\\ 
12.0 1.1145192544608108e14\\ 
12.5 3.5364993112563475e14\\ 
13.0 1.1087315395871004e15\\ 
13.5 3.6665650914016775e15\\ 
14.0 1.2184847626982606e16\\ 
};  \label{figcondnbs} 
\end{axis} 
\end{tikzpicture}%
\begin{tikzpicture}
\begin{axis}[%
width=2.5in,
height=1.5in,
at={(0.772in,0.516in)},
scale only axis,
xmin=0,
xmax=14,
xlabel style={font=\color{white!15!black}},
xlabel={$e$},
ymode=log,
ymin=1e-17,
ymax=1e-1,
yminorticks=true,
axis background/.style={fill=white}
]
\addplot [color=mycolor1, mark size=1.0pt, mark=*, mark options={solid, mycolor1}, dashed, forget plot]
table[row sep=crcr]{
0.0 1.7396958547654874e-15\\ 
0.5 1.6842375991805403e-14\\ 
1.0 6.108637558002145e-16\\ 
1.5 1.775594075883167e-15\\ 
2.0 2.927433751596372e-14\\ 
2.5 8.514050157560332e-14\\ 
3.0 1.4152249805269604e-13\\ 
3.5 2.3337953882076996e-12\\ 
4.0 5.2105010618508e-12\\ 
4.5 1.2831702168704035e-11\\ 
5.0 9.490131948142871e-12\\ 
5.5 2.8525963019224604e-10\\ 
6.0 1.5843170384431622e-10\\ 
6.5 1.2943535233076936e-10\\ 
7.0 1.606385153889906e-9\\ 
7.5 4.225648154722335e-9\\ 
8.0 1.6620682594731778e-7\\ 
8.5 1.1924868132659319e-7\\ 
9.0 9.013291620311531e-8\\ 
9.5 1.0742827547767855e-7\\ 
10.0 1.2893322632464077e-6\\ 
10.5 4.055425629613132e-6\\ 
11.0 1.4432423343762415e-5\\ 
11.5 0.00010144657400500874\\ 
12.0 0.00029038134652727436\\ 
12.5 4.054643026239699e-5\\ 
13.0 0.002315993157198248\\ 
13.5 0.00779057249645798\\ 
14.0 0.006863851715189287\\ 
};  \label{figmaxresCE} 
\addplot [color=mycolor1, mark size=1.0pt, mark=*, mark options={solid, mycolor1}, forget plot]
        table[row sep=crcr]{
0.0 0.0\\ 
0.5 1.8105721937103554e-15\\ 
1.0 4.62648332445569e-16\\ 
1.5 7.693233083441429e-16\\ 
2.0 1.9084112885645105e-15\\ 
2.5 2.2578000194176153e-15\\ 
3.0 1.0793334416089381e-14\\ 
3.5 1.2625961037685381e-13\\ 
4.0 5.252113729992582e-13\\ 
4.5 1.4246598299341501e-12\\ 
5.0 5.792676872832589e-12\\ 
5.5 4.610051366908456e-11\\ 
6.0 1.402998839811256e-11\\ 
6.5 4.530627833660608e-11\\ 
7.0 6.246676996220814e-11\\ 
7.5 1.023315582934372e-9\\ 
8.0 8.466278593454338e-9\\ 
8.5 2.506313203504278e-8\\ 
9.0 6.147101613413968e-9\\ 
9.5 4.6758102381579357e-8\\ 
10.0 2.5561269407825433e-7\\ 
10.5 1.7694850480119473e-6\\ 
11.0 7.626651474794553e-7\\ 
11.5 1.2924991307851585e-5\\ 
12.0 9.934799445972763e-6\\ 
12.5 1.4216279886868647e-5\\ 
13.0 0.0004841136427793718\\ 
13.5 0.0004649071941545169\\ 
14.0 0.0023253577405522845\\ 
};  \label{figmeanresCE} 
\addplot [color=mycolor1, mark size=1.0pt, dashed, mark=*, mark options={solid, mycolor1}, forget plot]
        table[row sep=crcr]{
0.0 0.0\\ 
0.5 5.316352693592218e-16\\ 
1.0 2.839751314938306e-16\\ 
1.5 2.119574710967524e-16\\ 
2.0 2.9413176141134487e-16\\ 
2.5 1.6751474634112117e-16\\ 
3.0 1.894246262063836e-15\\ 
3.5 2.0239229454430327e-14\\ 
4.0 1.1061303991057005e-13\\ 
4.5 2.8021193064623043e-13\\ 
5.0 2.8010559350726662e-12\\ 
5.5 1.1450210493035112e-11\\ 
6.0 2.7214585035020112e-12\\ 
6.5 1.6359895637677902e-11\\ 
7.0 7.706243573917144e-12\\ 
7.5 3.1084802470269455e-10\\ 
8.0 1.178625836876141e-9\\ 
8.5 7.085367407962164e-9\\ 
9.0 9.91617386939458e-10\\ 
9.5 1.9019092747559763e-8\\ 
10.0 7.018620897611907e-8\\ 
10.5 7.208013078116688e-7\\ 
11.0 1.0812483974461376e-7\\ 
11.5 2.8450215464737225e-6\\ 
12.0 1.1332179503319303e-6\\ 
12.5 5.19097929047048e-6\\ 
13.0 0.0001365193025970089\\ 
13.5 7.004366126951798e-5\\ 
14.0 0.0008337022389273052\\ 
};  \label{figminresCE} 
\addplot [color=mycolor2, mark size=1.0pt, dashed, mark=*, mark options={solid, mycolor2}, forget plot]
        table[row sep=crcr]{
0.0 7.585917874925291e-16\\ 
0.5 1.4121234155901038e-15\\ 
1.0 6.708322041905929e-16\\ 
1.5 7.610254721049285e-16\\ 
2.0 1.1602832768519269e-15\\ 
2.5 1.1757237795549336e-15\\ 
3.0 1.377395510793921e-15\\ 
3.5 9.44788411899365e-16\\ 
4.0 4.3001470553724095e-16\\ 
4.5 5.948316803303329e-16\\ 
5.0 2.3569837094520706e-15\\ 
5.5 9.234371610775305e-16\\ 
6.0 1.0882221640229961e-15\\ 
6.5 1.1003429454048666e-15\\ 
7.0 1.4491773331714915e-15\\ 
7.5 9.519280845455217e-16\\ 
8.0 2.469538065043172e-16\\ 
8.5 1.2378136093660367e-15\\ 
9.0 1.743339014582763e-15\\ 
9.5 1.880003215337906e-15\\ 
10.0 6.289064082159481e-16\\ 
10.5 1.7506977827926341e-15\\ 
11.0 1.2651436925685848e-15\\ 
11.5 6.013391656833934e-16\\ 
12.0 1.8247251869005407e-16\\ 
12.5 1.884225640475488e-15\\ 
13.0 1.5846452835427612e-15\\ 
13.5 3.125696226368417e-15\\ 
14.0 3.616946072008684e-15\\ 
};  \label{figmaxresBT} 
\addplot [color=mycolor2, mark size=1.0pt, dashed, mark=*, mark options={solid, mycolor2}, forget plot]
        table[row sep=crcr]{
0.0 8.585052531783988e-17\\ 
0.5 5.913307022361379e-16\\ 
1.0 6.958768337676864e-17\\ 
1.5 1.692734647387056e-16\\ 
2.0 3.1569473645776606e-17\\ 
2.5 1.1647054610444971e-16\\ 
3.0 4.0133363514206245e-16\\ 
3.5 1.7221010966172998e-16\\ 
4.0 1.2849651598124614e-16\\ 
4.5 2.8801438302596987e-16\\ 
5.0 1.5097725806169854e-16\\ 
5.5 7.787466346381119e-17\\ 
6.0 1.1027733060971003e-16\\ 
6.5 4.8015718205346983e-17\\ 
7.0 4.7536342069171046e-17\\ 
7.5 4.723061853955413e-16\\ 
8.0 9.745482529064883e-17\\ 
8.5 1.7659558394044294e-16\\ 
9.0 1.2650443706615904e-16\\ 
9.5 9.344728362510041e-17\\ 
10.0 1.3604665953066723e-16\\ 
10.5 1.2006536298806424e-16\\ 
11.0 9.133644188444683e-17\\ 
11.5 1.3480385765515902e-16\\ 
12.0 7.495120651851407e-17\\ 
12.5 4.407992350688243e-16\\ 
13.0 6.012119820423373e-17\\ 
13.5 4.2624595681780247e-16\\ 
14.0 9.953294682004521e-16\\ 
};  \label{figminresBT} 
\addplot [color=mycolor2, mark size=1.0pt, mark=*, mark options={solid, mycolor2}, forget plot]
        table[row sep=crcr]{
0.0 2.125625274719149e-16\\ 
0.5 8.037900803190549e-16\\ 
1.0 1.8661848570474325e-16\\ 
1.5 3.0244764273938807e-16\\ 
2.0 1.9879116364714406e-16\\ 
2.5 3.782154716974597e-16\\ 
3.0 6.159190828880737e-16\\ 
3.5 4.7824847260046115e-16\\ 
4.0 2.6164843380636205e-16\\ 
4.5 4.460781680055304e-16\\ 
5.0 4.977808017254346e-16\\ 
5.5 3.3093678214417166e-16\\ 
6.0 2.8824882429570815e-16\\ 
6.5 2.148730115042231e-16\\ 
7.0 3.7597226039736296e-16\\ 
7.5 7.072961284380233e-16\\ 
8.0 1.5475267003675053e-16\\ 
8.5 3.827972531937151e-16\\ 
9.0 5.541227864077816e-16\\ 
9.5 4.716725336777656e-16\\ 
10.0 2.595944991161434e-16\\ 
10.5 3.121399492757683e-16\\ 
11.0 2.4888204169547295e-16\\ 
11.5 3.155793335199621e-16\\ 
12.0 1.197513441599567e-16\\ 
12.5 7.932901342907855e-16\\ 
13.0 3.6686076262475675e-16\\ 
13.5 1.3158057348444604e-15\\ 
14.0 1.550626876616725e-15\\ 
};  \label{figmeanresBT} 
\end{axis} 
\end{tikzpicture}%
\caption{Left: norm of the largest solution (\ref{fignorms}) and condition number of the matrix inverted in the Canny-Emiris algorithm (\ref{figcondnbs}) as a function of $e$. Right: maximum, mean and minimum residual of all solutions computed by the Canny-Emiris algorithm (\ref{figminresCE} and \ref{figmeanresCE}) and by the algorithm based on this paper (\ref{figminresBT} and \ref{figmeanresBT}) in function of $e$.}
\label{fig:intro}
\end{figure}
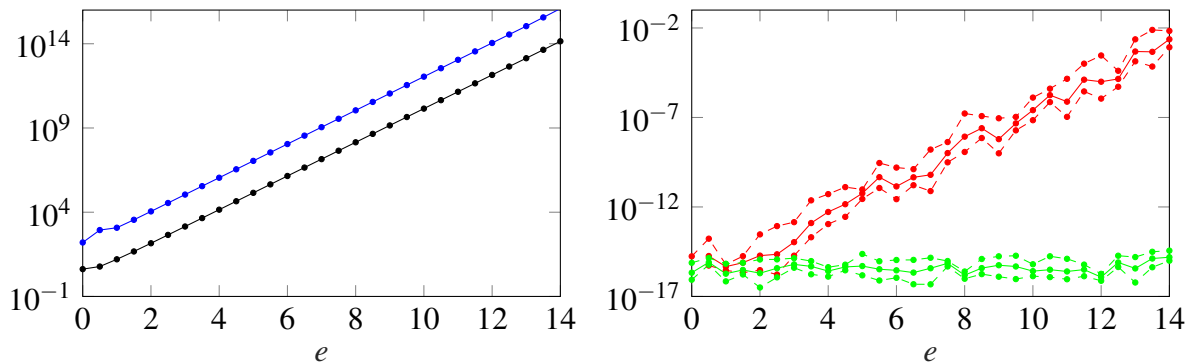
\end{example}
\paragraph{Outline.} The paper is organized as follows.
In \Cref{sec:prelim}, we review the notation and results on toric
varieties that we need in the rest of the paper.  
In \Cref{sec:fatpoints}, we prove our toric eigenvalue theorem and
detail our numerical algorithm for solving sparse polynomial systems.
In \Cref{sec:regularity}, we investigate the regularity of
zero-dimensional ideals and construct regularity pairs for complete
intersections on toric varieties.

\section{Preliminaries} \label{sec:prelim}
In this section we recall some facts and introduce some notation related to toric varieties, Cox rings and divisors. The reader who is unfamiliar with concepts from toric geometry can find more details in \cite{cox_toric_2011,fulton1993introduction}. 
To avoid confusion, for an ideal $I$ and a variety $X$ we write $V_X(I)$ for the \emph{subscheme} of $X$ defined by $I$ and $\Var_X(I)$ for the \emph{subvariety} of $X$ defined by $I$. That is, $\Var_X(I)$ is the reduced scheme associated to $V_X(I)$. 

\subsection{Toric geometry and the Cox construction}
 \label{subsec:toric}
We write $T = (\C^*)^n = (\C \setminus \{0\})^n$ for the algebraic torus with character lattice $M = \Hom_\Z(T,\C^*) \simeq \Z^n$ and cocharacter lattice $N = M^\vee = \Hom_\Z(\C^*,T)$.
A rational polyhedral fan $\Sigma$ in $N \otimes_\Z \R = N_\R \simeq \R^n$ defines a normal toric variety $X_\Sigma$. 
This variety is complete (or equivalently, compact) if and only if the support of $\Sigma$ is $N_\R$, in which case we also call $\Sigma$ complete. 
In the rest of this paper, $X = X_\Sigma$ is a normal toric variety corresponding to a complete fan $\Sigma$, and we will sometimes write the subscript $\Sigma$ to emphasize this correspondence. 

The toric variety $X_\Sigma$ admits an affine open covering given by the affine toric varieties $\{ U_\sigma ~|~ \sigma \in \Sigma \}$ corresponding to the cones of $\sigma$. These affine varieties are defined as follows. For each cone $\sigma \in \Sigma$, the dual cone $\sigma^\vee \subset M_\R = M \otimes_\Z \R \simeq \R^n$ gives a saturated semigroup $\sigma^\vee \cap M$ whose associated $\C$-algebra $\C[\sigma^\vee \cap M] = \C[U_\sigma]$ is the coordinate ring of the affine toric variety $U_\sigma$. The way these affine varieties are glued together to obtain $X_\Sigma$ is encoded by the fan $\Sigma$, see \cite[Ch.~3]{cox_toric_2011}.

Let $\Sigma(d)$ be the set of $d$-dimensional cones of $\Sigma$. In particular, the rays of $\Sigma$ are $\Sigma(1) = \{\rho_1, \ldots, \rho_k \}$. They correspond to the torus invariant divisors $D_1, \ldots, D_k$ on $X_\Sigma$, which generate the free group $\Div_T(X_\Sigma) \simeq \Z^k$. Each $\rho_i \in \Sigma(1)$ has a unique primitive ray generator $u_i \in N$. It is convenient to collect the $u_i$ in a matrix 
$$ F = [u_1 ~ \cdots ~ u_k] \in \Z^{n \times k}.$$
The \emph{divisor class group} $\Cl(X_\Sigma)$ is isomorphic to $\Div_T(X_\Sigma)/\im\, F^\top$, so it is generated by the equivalence classes $[D_i]$, $i = 1, \ldots, k$, see \cite[Ch.~ 4]{cox_toric_2011}.

In \cite{cox1995homogeneous}, Cox shows that $X_\Sigma$ can be realized as a GIT quotient of a quasi-affine space by the action of an algebraic reductive group. The quotient is given by a surjective toric morphism 
\begin{equation} \label{eq:coxquotient} 
\pi: \C^k \setminus Z \rightarrow X_\Sigma,
\end{equation}
where $\C^k$ is the \emph{total coordinate space} of $X$ with coordinates labeled by $\Sigma(1)$ and the variety $Z = \Var_{\C^k}(B)$ is the \emph{base locus}, given by the zero set of the \emph{irrelevant ideal} $B \subset S = \C[x_1,\ldots,x_k]$. This is the square-free monomial ideal $B = \ideal{ x^{\hat{\sigma}} ~|~ \sigma \in \Sigma }$, where $x^{\hat{\sigma}} = \prod_{\rho_i \notin \sigma} x_i$. 
The map $\pi$ is constant on the orbits of the action of an algebraic reductive subgroup $G \subset (\C^*)^k$, which acts on $\C^k \setminus Z$ by restricting the natural $(\C^*)^k$-action.

For an element $\alpha = [\sum_{i=1}^k a_i D_i] \in \Cl(X_\Sigma)$, we define the vector subspace 
$$ S_\alpha = \bigoplus_{F^\top m + a \geq 0} \C \cdot x^{F^\top m + a} \quad \subset S,$$
where $a = (a_1, \ldots, a_k) \in \Z^k$ and the sum ranges over all $m \in M$ satisfying $\pair{u_i,m} + a_i \geq 0, i = 1, \ldots, k$. This definition is independent of the chosen representative for $\alpha$.
The action of $G$ on $\C^k$ induces an action of $G$ on $S$: for $g \in G$, $f \in S$, $(g \cdot f)(x) = f(g^{-1} \cdot x)$, and an element $f \in S_\alpha$ defines an affine subvariety $\Var_{\C^k}(f)$ that is stable under the action of $G$. From this observation, it follows that an element $f \in S_\alpha$ has a well-defined zero set on $X_\Sigma$, given by 
$$ \Var_{X_\Sigma}(f) = \{ \z \in X_\Sigma ~|~ f(x) = 0 \textup{ for some } x \in \pi^{-1}(\z) \}.$$
This is why $S$ is equipped with its  \emph{grading by the class group}: $S = \bigoplus_{\alpha \in \Cl(X_\Sigma)} S_\alpha$. The ring $S$, together with this grading and its irrelevant ideal, is called the \emph{Cox ring}, homogeneous coordinate ring or total coordinate ring of $X_\Sigma$.
The homogeneous ideals of $S$, that is, the ideals
generated by elements that are homogeneous with respect to the
$\Cl(X_\Sigma)$-grading, define the closed subschemes of $X_\Sigma$, see \cite[Ch.~5 \& 6]{cox_toric_2011}. If $X_\Sigma$ is smooth and we restrict to $B$-saturated homogeneous ideals, this correspondence is one-to-one \cite[Cor.~3.8]{cox1995homogeneous}. For a homogeneous ideal $I \subset S$, the corresponding subscheme is denoted by $V_{X_\Sigma}(I)$ and its associated variety is $\Var_{X_\Sigma}(I)$. 
\paragraph{Solving equations on $X$.}
Given homogeneous polynomials $f_1, \ldots, f_s \in S$, such that the associated scheme $V_{X_\Sigma}(I)$ is zero-dimensional, the aim in this paper is to \emph{solve} $f_1 = \cdots = f_s =0$ on $X$. We will now make this precise. For each $\z_i \in \Var_X(I)$, we want to compute a point $z_i \in \C^k \setminus Z$ such that $\pi(z_i) = \z_i$. In this context, the point $z_i$ is called a \emph{set of homogeneous coordinates} for $\z_i$.

\subsection{Homogenization and dehomogenization} \label{subsec:hom}

\paragraph{Homogenization.} The $\C$-algebra $\C[M]$ over the lattice $M$ is isomorphic to the ring $\C[t_1^{\pm 1}, \ldots, t_n^{\pm 1}]$ of $n$-variate Laurent polynomials. We consider $s$ elements $\f_1, \ldots, \f_s$ of $\C[M]$. These elements define a system of relations $\f_1 = \ldots = \f_s = 0$ on $T$, which extends to a system of relations on a toric compactification $X \supset T$. We will make this precise in this subsection.

For each $\f_i$, let $P_i \subset M_\R = \R^n$ be its Newton
polytope, i.e., the convex hull in $M_\R$ of the characters appearing
in $\f_i$ with a nonzero coefficient. Let $P = P_1 + \ldots + P_s$ be
the Minkowski sum of all these polytopes. We assume that $P$ is
full-dimensional. The normal fan $\Sigma_P$ of $P$ defines a complete,
normal toric variety $X = X_{\Sigma_P}$. We will use the same notation
as in \Cref{subsec:toric} for the rays, primitive ray generators,
etc. To each of the polytopes $P_i$, we associate a torus invariant
divisor $D_{P_i}\in \Div_T(X)$ as follows. Let
$a_i = (a_{i,1}, \ldots, a_{i,k}) \in \Z^k$ be such that
\begin{align*}
a_{i,j} = \min_\Z c 
\textup{~ s.t.~ } P_i \subset \{ m \in M_\R ~|~ \pair{u_j,m} + c \geq 0 \}.
\end{align*}
We set $D_{P_i} = \sum_{j=1}^k a_{i,j} D_j \in \Div_T(X)$. With this construction, the $D_{P_i}$ are \textit{Cartier
  divisors}. The classes of all Cartier divisors in $\Cl(X)$ form a group called the \textit{Picard
  group} $\Pic(X) \subset \Cl(X)$.
We have
that $\alpha_i = [D_{P_i}] \in \Pic(X)$
and, additionally, the divisors $D_{P_i}$ are \emph{basepoint free} (see
\Cref{def:bpf}).

We start by `homogenizing' the $\f_i$ to the Cox ring $S$ of $X$. For this, we observe that by construction 
$$ \f_i \in \bigoplus_{m \in P_i \cap M} \C \cdot t^m \simeq \bigoplus_{F^\top m + a_i \geq 0} \C \cdot t^m \simeq \bigoplus_{F^\top m + a_i \geq 0} \C \cdot x^{F^\top m + a_i} = S_{\alpha_i}.$$
This gives a canonical way of \emph{homogenizing} $\f_i$: 
\begin{equation} \label{eq:homogenization}
 \f_i = \sum_{F^\top m + a_i \geq 0} c_{m,i} t^m \mapsto f_i = \sum_{F^\top m + a_i \geq 0} c_{m,i} x^{F^\top m + a_i}.
 \end{equation}
The subvariety $\Var_X(f_i) \subset X$ is the closure of $\Var_T(\f_i)$ in $X$.
\begin{example} \label{ex:doublepillow1}
Let $n = 2$, $\C[M] = \C[t_1^{\pm 1 }, t_2^{\pm 1} ]$ and consider the equations 
\begin{align*}
\f_1 &= t_1 - t_2^{-1} + t_2 + t_1^{-1} , \quad  \f_2 = 2t_1 + t_2^{-1} -t_2 -t_1^{-1} .
\end{align*}
There are no solutions of $\f_1 = \f_2 = 0$ in $(\C^*)^2$ (note that $\f_1 + \f_2$ is a unit in $\C[M]$). The Newton polygons $P_1, P_2$ are identical. The associated toric variety $X$ is the \emph{double pillow surface} (see \cite[Sec.~3.3]{sottile2017ibadan}).
The fan $\Sigma = \Sigma_{P_1+P_2}$ is depicted in Figure \ref{fig:doublepillow}. 
\begin{figure}
\centering
\begin{tikzpicture}[baseline=-1 cm,scale=.75]
	\coordinate (O) at (0,0);
	\coordinate (u1) at (1,1);
	\coordinate (uu1) at (2,2);
	\coordinate (u2) at (-1,1);
	\coordinate (uu2) at (-2,2);
	\coordinate (u3) at (-1,-1);
	\coordinate (uu3) at (-2,-2);
	\coordinate (u4) at (1,-1);
	\coordinate (uu4) at (2,-2);
	\coordinate (X) at (1.5,0);
	\coordinate (Y) at (0,1.5);
	\draw[-latex] (O) -> (u1); 
	\draw[-latex] (O) -> (u2);
	\draw[-latex] (O) -> (u3); 
	\draw[-latex] (O) -> (u4);
	\draw[opacity=0.2,fill = blue,] (O)--(uu1)--(uu2)--cycle;
	\draw[opacity=0.2,fill = green,] (O)--(uu2)--(uu3)--cycle;
	\draw[opacity=0.2,fill = yellow,] (O)--(uu3)--(uu4)--cycle;
	\draw[opacity=0.2,fill = orange,] (O)--(uu4)--(uu1)--cycle;
	\draw (0.6,-1) node{};
	\draw (-1,0.5) node{};
	\draw (1,1) node{};
	\draw (0,1.3) node{$\sigma_{1}$};
\end{tikzpicture} 
\qquad
\begin{tikzpicture}[baseline=0.5 cm,scale=1.2]
	\coordinate (O) at (0,0);
	\coordinate (u1) at (1,1);
	\coordinate (uu1) at (2.5,2.5);
	\coordinate (u2) at (-1,1);
	\coordinate (uu2) at (-2.5,2.5);

	\draw[-latex, very thick] (O) -> (u1); 
	\draw[-latex, very thick] (O) -> (u2);
	\draw[opacity=0.2,fill = blue,] (O)--(uu1)--(uu2)--cycle;
	\draw (1.3,1) node{\footnotesize $y_3$};
	\draw (0.2,1) node{\footnotesize $y_2$};
	\draw (-1.3,1) node{\footnotesize $y_1$};
	\draw (0,2.2) node{\footnotesize $y_2^2=y_1y_3$};
	
	\draw (1,1) node{};
	\filldraw (0,0) circle (1pt);
	\filldraw (1,1) circle (1pt);
	\filldraw (-1,1) circle (1pt);
	\filldraw (0,1) circle (1pt);
	\filldraw (2,2) circle (1pt);
	\filldraw (0,2) circle (1pt);
	\filldraw (-2,2) circle (1pt);
	\filldraw (-1,2) circle (1pt);
	\filldraw (1,2) circle (1pt);
\end{tikzpicture} 
\caption{Illustration of the fan $\Sigma_{P_1+P_2}$ (left) of the toric variety from Example \ref{ex:doublepillow1}, and of the semigroup algebra $\C[U_{\sigma_1}] \simeq \C[y_1,y_2,y_3]/\ideal{y_2^2-y_1y_3}$ corresponding to the (dual cone of the) blue cone (right).}
\label{fig:doublepillow}
\end{figure}
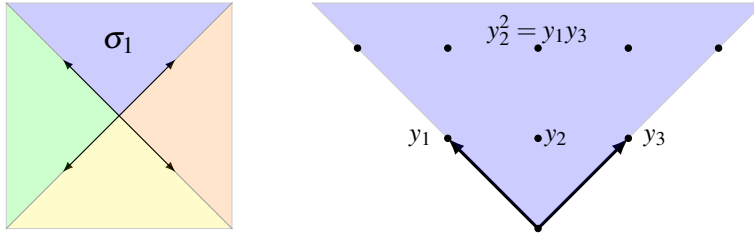
We arrange the primitive ray generators of $\Sigma(1)$ in the matrix 
$ F = [u_1 ~ u_2 ~ u_3 ~ u_4] = \left [ \begin{smallmatrix}
1&-1&-1&1\\1&1&-1&-1
\end{smallmatrix} \right ].$
Our equations homogenize to $
f_1 = x_1^2x_4^2 - x_3^2x_4^2 + x_1^2x_2^2 + x_2^2x_3^2,~f_2= 2x_1^2x_4^2 + x_3^2x_4^2 - x_1^2x_2^2 - x_2^2x_3^2$
in the Cox ring $S= \C[x_1,x_2,x_3,x_4]$ of $X$. The degrees $\alpha_i = \deg(f_i)$ are $\alpha_1 = \alpha_2 = [\sum_{i=1}^4 D_i]$. The scheme $V_X(f_1,f_2)$ consists of two points, each with multiplicity two. These points correspond to the orbits of 
$ z_1 = (0,1,1,1),  z_2 = (1,1,0,\sqrt{-1})$. 
\end{example}

\paragraph{Dehomogenization.} Recall that $X_\Sigma$ is covered by the affine toric varieties $\{ U_\sigma ~|~ \sigma \in \Sigma(n) \}$. The restriction of \eqref{eq:coxquotient} to $\pi^{-1}(U_\sigma) = \C^k \setminus \Var_{\C^k}(x^{\hat{\sigma}})$ identifies $\C[U_\sigma]$ with the ring of invariants $(S_{x^{\hat{\sigma}}})_0$, see the proof of \cite[Thm. 5.1.10]{cox_toric_2011}. For each full-dimensional cone $\sigma \in \Sigma(n)$, we will define a $\C$-linear map
$ (\cdot)^\sigma: S \rightarrow \left (S_{x^{\hat{\sigma}}} \right )_0 \simeq \C[U_\sigma]$, called \emph{dehomogenization}. We do this by defining it on graded pieces $S_\alpha$, and extending linearly. For $\alpha = [\sum_{i=1}^k a_iD_i ] $ such that there exists $m_{\sigma} \in M$ with
$\pair{u_i,m_{\sigma}} + a_i = 0$ for all $i$ such that $\rho_i \in \sigma(1)$,
we set 
$$f \in S_\alpha \mapsto f^\sigma := \frac{f}{x^{\hat{\sigma},\alpha}}, \qquad \text{ with } \quad x^{\hat{\sigma},\alpha} := x^{F^\top m_{\sigma} + a}  ~ \in S_\alpha.$$ 
Note that, although $m_{\sigma}$ depends on the choice of representative $\sum_{i=1}^k a_iD_i$ for $\alpha$, the monomial $x^{\hat{\sigma},\alpha}$ does not. This is the only monomial of degree
$\alpha$ such that $(x^{\hat{\sigma},\alpha})^\sigma = 1$. For $\alpha \in \Cl(X)$ for which no such $m_{\sigma}$ exists, we set $(S_\alpha)^\sigma = 0$.
Observe that, if $(S_\alpha)^\sigma \neq 0$, the restriction of the dehomogenization map to $S_\alpha$ is injective. In particular, this is the case for
$\alpha \in \Pic(X)$ \cite[Thm.~4.2.8]{cox_toric_2011}.
For each $\sigma \in \Sigma(n)$, dehomogenization $(\cdot)^\sigma: S \rightarrow \C[U_\sigma]$ is surjective. Moreover, for a homogeneous ideal $I \subset S$ we have
\begin{equation} \label{eq:dehomideal}
I^\sigma = (I_{x^{\hat{\sigma}}})_0 = \I(U_\sigma) \subset \C[U_\sigma],
\end{equation}
where $\I$ is the ideal sheaf associated to $I$. 
\begin{example}[Cont. Example \ref{ex:doublepillow1}]  \label{ex:doublepillow1b}
Let $\sigma_1$ be the blue cone in Figure \ref{fig:doublepillow}. The ideal $I^{\sigma_1} = \ideal{f_1^{\sigma_1},f_2^{\sigma_1}} \subset \C[U_{\sigma_1}] = \C[y_1,y_2,y_3]/\ideal{y_2^2-y_1y_3}$ is $ I^{\sigma_1} = \ideal{[y_2^2-y_1+y_3+1], [2y_2^2+y_1-y_3-1]},$
where $[\cdot]$ denotes the residue class modulo $\ideal{y_2^2-y_1y_3}$.
The ordering of the variables $y_i$ is clarified in the right part of Figure \ref{fig:doublepillow}. Only the solution corresponding to the orbit of $z_1$ is contained in $U_{\sigma_1}$, which explains that $\dim_\C \C[U_{\sigma_1}]/I^{\sigma_1}= 2$. 
\end{example}
The following lemma points out a way of going back and forth between $I$ and $I^\sigma$.
\begin{lemma}
  \label{lem:clearDenom}
  Suppose that $f \in S$ is such that $f^\sigma \in I^\sigma \setminus \{0\}$, then there is $\ell \in \N$ such that
  $(x^{\hat{\sigma}})^\ell \, f \in I$.
  Conversely, for each
  $\f \in I^\sigma$, we can find a homogeneous $f \in I$ such that
  $f^\sigma = \hat{f}$.
\end{lemma}
\begin{proof}
The first statement follows immediately from $I^\sigma = (I_{x^{\hat{\sigma}}})_0$. For the second statement, note that $\f$ can be written as $\f = f/(x^{\hat{\sigma}})^\ell$ with $f \in I_{\ell \deg(x^{\hat{\sigma}})}$. We have $x^{\hat{\sigma}, \ell \deg(x^{\hat{\sigma}})} = (x^{\hat{\sigma}})^\ell$ and thus $((x^{\hat{\sigma}})^\ell  \, \hat{f} )^\sigma = f^\sigma =  \hat{f}$.
\end{proof}

\subsection{Divisors}
\label{subsec:divisors}
\noindent It will be convenient to have a notation for subsets of the divisor class group that are of interest:
\begin{center}
\begin{tabular}{l|l}
$\Cl(X)$      & divisor class group of $X$                                         \\
$\Cl(X)_+$    & divisor classes of \emph{effective} divisors:$\{\alpha \in \Cl(X) ~|~ \alpha = [\sum_{i=1}^k a_i D_i], \, a_i \geq 0 \}$          
\end{tabular}
\end{center}
We will work with a subclass of these divisors given by the
\emph{basepoint free} divisors.
\begin{definition} \label{def:bpf}
For a fixed degree $\alpha \in \Cl(X)$, a point $\z \in X$ is called a \emph{basepoint} of $S_\alpha$ (or of $\alpha$) if $\z \in V_X(f)$ for all $f \in S_\alpha$. The degree $\alpha \in \Cl(X)$ is called \emph{basepoint free} if it has no basepoints. A torus invariant divisor $D \in \Div_T(X)$ is called \emph{basepoint free} if $[D] \in \Cl(X)$ is basepoint free.
\end{definition}
We introduce a similar notation for subsets of the Picard group, and
the (in general larger) subgroup of $\Cl(X)$ consisting of the divisor
classes of \emph{$\Q$-Cartier} divisors on $X$. Note that, as $X$ is a
complete toric variety, nef Cartier divisors correspond to basepoint
free Cartier divisors \cite[Thm.~6.3.12]{cox_toric_2011} and nef
$\Q$-Cartier divisors correspond to divisors of which a multiple is nef and Cartier \cite[Lem.~9.2.1]{cox_toric_2011}.
We make the following definitions.
    $$
      \begin{array}{l  | l }
\Pic(X)      & \text{Picard group of  } X \\
\Picbpf(X)   & \text{divisor classes of \emph{nef Cartier} divisors: }\\ & 
                   \{ \alpha \in \Pic(X) ~|~ \alpha \text{ is basepoint free }\} \\
\Q \Pic(X) & \text{divisor classes of \emph{$\Q$-Cartier} divisors: } \\ &  \{ \alpha \in \Cl(X) ~|~ \ell \alpha \in \Pic(X) \text{ for some } \ell \in \N_{>0} \}\\
\Q \Picbpf(X) & \text{divisor classes of \emph{nef $\Q$-Cartier} divisors:}  \\ &  \{ \alpha \in \Cl(X) ~|~ \ell \alpha \in \Picbpf(X) \text{ for some } \ell \in \N_{>0} \} 
      \end{array}
      \qquad      \vrule \qquad
      \begin{matrix}
\Picbpf(X) & \subset & \Pic(X)\\
\cap &  & \cap\\
\Q \Picbpf(X) &  \subset & \Q\Pic(X)\\
 \cap  & & \cap\\
 \Cl(X)_+ & \subset & \Cl(X)
\end{matrix}
$$
The diagram of inclusions on the right follows directly from these definitions. Note that if $\alpha \in \Cl(X) \setminus \Cl(X)_+$, we have that $S_{\alpha} = \{0\}$, and hence every $\z \in X$ is a basepoint of $S_\alpha$.

\begin{example} \label{ex:divisors}
Consider again the double pillow surface $X$ from Example \ref{ex:doublepillow1}.
Its fan is depicted in Figure \ref{fig:doublepillow}. One can check that $\Q \Pic(X) = \Cl(X)$ (which is true whenever $X$ is simplicial), $\alpha_0 := [D_2 + D_4] \in \Q \Picbpf(X) \setminus \Picbpf(X)$, $2\alpha_0 \in \Picbpf(X)$, $-\alpha_0 \in \Pic(X) \setminus \Cl(X)_+$. Moreover, $S_{\alpha_0}$ is not basepoint free: $S_{\alpha_0} = \C \cdot x_2x_3$ vanishes on $V_X(x_2) \cup V_X(x_3)$. 
\end{example}

We generalize the concept of basepoint free divisors to subschemes of $X$.

\begin{definition}
  Given a subscheme $Y$ of $X$, we say that $\alpha \in S$ is
  \emph{$Y$-basepoint free} if $Y_{\text{red}}$ does not contain any
  basepoint of $\alpha$, where $Y_{\text{red}}$ is the variety
  (reduced scheme) associated to $Y$. Whenever $Y$ is
  zero-dimensional, this condition is equivalent to the existence of
  $h \in S_\alpha$ such that
  $Y_{\text{red}} \cap \Var_X(h) = \emptyset$.
\end{definition}

\section{A toric eigenvalue method} \label{sec:fatpoints}

In this section we consider a homogeneous ideal $I = \ideal{f_1, \ldots, f_s} \subset S$ which defines a zero-dimensional, possibly \emph{non-reduced} subscheme $V_X(I)$. That is, some of the points in the scheme have multiplicity greater than one. 

We will give an explicit description of the toric eigenstructure of
multiplication maps (see \eqref{eq:multMap}) in the presence of non-reduced points. First, in
\Cref{subsec:affinefat} we recall what happens in the affine case. The
approach is similar to
\cite{marinari1993grobner,moller1995multivariate}, in the specific
case of a zero-dimensional subvariety of an affine toric variety. This
fixes some notation and sets the stage for the homogeneous case. In
\Cref{subsec:fatpoints}, we characterize the eigenstructure
of multiplication maps by `gluing' the affine
constructions. The eigenvectors have a natural interpretation as elements of the dual of the Cox ring. Finally, in \Cref{subsec:algorithm} we present an
eigenvalue approach for solving $f_1 = \cdots = f_s = 0$ on $X$, that is, for computing homogeneous coordinates of the
points in $\Var_X(I)$. Our algorithm also allows to compute the multiplicities of these points in $V_X(I)$.

Throughout the section, for a $\C$-vector space $V$, we denote
$V^\vee = \Hom_\C(V,\C)$ for the dual vector space.

\subsection{Non-reduced points on an affine toric variety} \label{subsec:affinefat}
Let $I^\sigma \subset \C[U_\sigma]$ be an ideal in the coordinate ring of the normal affine toric variety $U_\sigma$ coming from a full dimensional cone $\sigma \subset N_\R$. Let $\AA = \{m_1, \ldots, m_\ell\} \subset M$ be a set of characters such that $\sigma^\vee \cap M = \N \cdot \{m_1, \ldots, m_\ell\}$. The ring $\C[U_\sigma]$ can be realized as a quotient ring $\C[U_\sigma] \simeq \C[y_1, \ldots, y_\ell]/I_\AA$, where the variables $y_i$ correspond to the characters $m_i$ and $I_\AA$ is the toric ideal corresponding to the embedding of $U_\sigma$ in $\C^\ell$ given by $\AA$ \cite[Ch.~1]{cox_toric_2011}. Therefore, the ideal $I^\sigma$ corresponds to an ideal in $\C[y_1, \ldots, y_\ell]/I_\AA$, which in turn corresponds to an ideal in $\C[y_1, \ldots, y_\ell]$ containing $I_\AA$. In other words, the subscheme of $U_\sigma$ defined by $I^\sigma$ is embedded in $\C^\ell$
via 
$$\C[y_1, \ldots, y_\ell] \rightarrow \C[y_1, \ldots, y_\ell]/I_\AA \simeq \C[U_\sigma] \rightarrow \C[U_\sigma]/I^\sigma.$$
We will think of $V_{U_\sigma}(I^\sigma)$ as a subscheme of $\C^\ell$ given by this embedding and denote the defining ideal by $I^\sigma_y$. That is, 
$$V_{U_\sigma}(I^\sigma) = \Spec(\C[y_1, \ldots, y_\ell]/I^\sigma_y) = \Spec(\C[U_\sigma]/I^\sigma)$$ and $I_\AA \subset I^\sigma_y \subset \C[y_1, \ldots, y_\ell]$. The reason why we embed $V_{U_\sigma}(I_\sigma)$ in $\C^\ell$ is that the multiplicity structure of non-reduced points in affine space have a nice, explicit description in terms of differential operators. We will now recall how this works. For every $\zeta \in \C^\ell$, let $\m_\z \subset \C[y_1, \ldots, y_\ell]$ be the corresponding maximal ideal. Assuming that $V_{U_\sigma}(I^\sigma)$ is zero-dimensional, defining $\D_\sigma$ points $\{\zeta_1, \ldots, \zeta_{\D_\sigma} \} \subset \C^\ell$, the primary decomposition of $I^\sigma_y$ is given by $ I^\sigma_y = Q_1 \cap \cdots \cap Q_{\delta_\sigma}$,
where $Q_i$ is $\m_{\z_i}$-primary. By the Chinese remainder theorem \cite[Lem.~3.7.4]{kreuzer_computational_2000}  this gives 
\begin{equation} \label{eq:algdecomp}
\C[y_1, \ldots, y_\ell]/I^\sigma_y =~ \C[y_1, \ldots, y_\ell]/Q_1 ~\oplus~ \cdots ~\oplus~ \C[y_1, \ldots, y_\ell]/Q_{\delta_\sigma}.
\end{equation}
The multiplicities $\mu_i$ of the points $\zeta_i$ are given by 
$ \mu_i = \dim_\C \left (  \C[y_1, \ldots, y_\ell]/Q_i   \right )$.
We denote $\DD_\sigma = \mu_1 + \cdots + \mu_{\D_\sigma} = \dim_\C(\C[y_1, \ldots, y_\ell]/I^\sigma_y)$. 
For an $\ell$-tuple $a = (a_1, \ldots, a_\ell) \in \N^\ell$ we define the $\C$-linear map 
$ \partial_a: \C[y_1, \ldots, y_\ell] \rightarrow \C[y_1, \ldots, y_\ell]$ as
$$ \partial_a(f) = \frac{1}{a_1!\cdots a_\ell!}\frac{\partial^{a_1 + \cdots + a_\ell }f}{\partial y_1^{a_1} \cdots \partial y_\ell^{a_\ell}}$$
and $\Diff = \Span_\C(\partial_a, a \in \N^\ell)$.
These operators allow for a very simple formulation of Leibniz' rule, which says that for $\partial \in \Diff$, 
\begin{equation} \label{eq:leibniz}
\partial(fg) = \sum_{b \in \N^\ell} \partial_b(g)(s_b(\partial))(f) \quad \text{with} \quad s_b \left (\sum_{a} c_a \partial_a \right) = \sum_{a-b \geq 0} c_{a} \partial_{a-b}.
\end{equation}
\begin{definition}
A $\C$-vector subspace $V \subset \Diff$ is \emph{closed} if $\dim_\C(V) < \infty$ and for each $\partial \in V$ and each $b \in \N^\ell$, $s_b(\partial) \in V$.
\end{definition}
Note that if $V \subset \Diff$ is closed, then $\partial_0 = \id \in V$. 
For $\zeta \in \C^\ell$, let
$\eval_\zeta \in \Hom_\C(\C[y_1, \ldots, y_\ell], \C) = \C[y_1,
\ldots, y_\ell]^\vee$ be the linear functional defined by
$\eval_\zeta(f) = f(\zeta)$. It follows from \cite[Thm.~2.6]{marinari1993grobner} that, for each
$i \in \{1, \ldots, \D_\sigma\}$, there is a closed subspace
$V_i \subset \Diff$ such that, 
$$ Q_i = \{ f \in \C[y_1, \ldots, y_\ell] ~|~ (\eval_{\zeta_i} \circ \partial)(f) = \partial(f)(\zeta_i) = 0, \forall \partial \in V_i \}.$$
For each $Q_i$, we define the set of linear functionals
$Q_i^\perp$ as follows,
$$Q_i^\perp := \{ v \in \C[y_1, \ldots, y_\ell]^\vee ~|~ v(f) = 0, \forall f \in Q_i \} = \{ \eval_{\zeta_i} \circ \partial ~|~ \partial \in V_i \} =: \eval_{\zeta_i} \circ V_i.$$
From basic linear algebra,
it follows that
$$\eval_{\zeta_i} \circ V_i = Q_i^\perp \simeq (\C[y_1, \ldots,
y_\ell]/Q_i)^\vee \subset (\C[y_1, \ldots, y_\ell]/I^\sigma_y)^\vee.$$
In other words, we can interpret elements of $\eval_{\z_i} \circ V_i$ as elements of
$(\C[y_1, \ldots, y_\ell]/I^\sigma_y)^\vee$, 
by setting $(\eval_{\zeta_i} \circ \partial)(f + I^\sigma_y) = (\eval_{\zeta_i} \circ \partial)(f)$ for $\partial \in V_i$.

For $g \in \C[y_1, \ldots, y_\ell]$, the \emph{multiplication map} $M_g : \C[y_1, \ldots, y_\ell]/I^\sigma_y \rightarrow \C[y_1, \ldots, y_\ell]/I^\sigma_y$ is defined as 
$ M_g( f + I^\sigma_y) = fg + I^\sigma_y$.
For a differential operator $\partial = \sum_a c_a \partial_a \in \Diff$ we define $\ord(\partial) = \max_{c_a \neq 0} (a_1 + \cdots + a_\ell)$. We denote by $(V_i)_{\leq d} = \{ \partial \in V_i ~|~ \ord(\partial) \leq d \}$ the subspace of differential operators in $V_i$ of order bounded by $d$. 

For giving explicit descriptions of the eigenstructure of multiplication maps, it is convenient to work with a special type of basis for the spaces $V_i$ (see \cite[Sec.~5]{moller1995multivariate}). 
Let $V \subset \Diff$ be a closed subspace. An ordered tuple $(\partial^{1}, \ldots, \partial^{\mu})$ with $\partial^{j} \in V, j = 1, \ldots, \mu$ is called a \emph{consistently ordered basis} for $V$ if for every $d \geq 0$ there is $j_d$ such that $\{\partial^{1}, \ldots, \partial^{j_d} \}$ is a $\C$-vector space basis for $V_{\leq d}$. 
Note that a consistently ordered basis always exists for any closed
subspace $V$, its first differential operator is always $\partial_0 = \id$
and it is a $\C$-vector space basis for $V$.
For $i = 1, \ldots, \D_\sigma$, let $(\partial^{i1}, \ldots, \partial^{i\mu_i})$ be a consistently ordered basis for $V_i$. By Leibniz' rule, for all $f + I^\sigma_y \in \C[y_1, \ldots, y_\ell]/I^\sigma_y$ we have
\begin{equation} \label{eq:multmtx}
((\eval_{\zeta_i} \circ \partial^{ij}) \circ M_g) (f + I^\sigma_y) = \eval_{\zeta_i}(\partial^{ij}(fg)) = (\eval_{\zeta_i} \circ \sum_{b \in \N^\ell} \partial_b(g) s_b(\partial^{ij}))(f+ I^\sigma_y).
\end{equation}
In particular, for $\partial^{i1} = \partial_0 = \id$ we get 
$ \eval_{\zeta_i} \circ M_g = g(\zeta_i) \, \eval_{\zeta_i}$,
which shows the classical fact that the evaluation functionals $\eval_{\zeta_i}$ are (left) eigenvectors of $M_g$ with eigenvalues $g(\zeta_i)$. In general, by the property of being closed, $s_b(\partial^{ij})$ can be written as a $\C$-linear combination of $\partial^{i1}, \ldots, \partial^{i \mu_i}$. For $b \neq 0$, by the property of being consistently ordered and by the fact that $\ord(s_b(\partial)) < \ord(\partial)$, $s_b(\partial^{ij})$ can be written as a $\C$-linear combination of $\partial^{i1}, \ldots, \partial^{i,j-1}$ (in fact, we only need the differentials of order $< \ord(\partial^{ij})$). 
Then, in matrix notation, \eqref{eq:multmtx} becomes 
\begin{equation} \label{eq:mateq1}
\begin{bmatrix}
\eval_{\zeta_i} \circ \partial^{i1} \\ \eval_{\zeta_i} \circ \partial^{i2} \\ \vdots \\ \eval_{\zeta_i} \circ \partial^{i \mu_i}
\end{bmatrix} \circ M_g = \begin{bmatrix}
g(\zeta_i) \\ c_{i2}^{(1)} & g(\zeta_i) \\ \vdots & & \ddots \\ 
c_{i \mu_i}^{(1)} & c_{i \mu_i}^{(2)} & \hdots & g(\zeta_i)
\end{bmatrix} \begin{bmatrix}
\eval_{\zeta_i} \circ \partial^{i1} \\ \eval_{\zeta_i} \circ \partial^{i2} \\ \vdots \\ \eval_{\zeta_i} \circ \partial^{i \mu_i}
\end{bmatrix}
\end{equation}
for some complex coefficients $c_{ij}^{(k)}$. This observation leads immediately to the classical eigenvalue theorem \cite[Ch.~4, \S 2, Prop.~2.7]{cox_using_2005} and it will play a key role in the proof of Theorem \ref{thm:toriceval}.

\subsection{Toric eigenvalue theorem} \label{subsec:fatpoints}

In what follows, we fix a homogeneous ideal $I \subset S$ defining a
zero-dimensional closed, possibly \emph{non-reduced}, subscheme
$V_X(I)$.
As before, we write $\{\zeta_1, \ldots, \zeta_\D\} \subset X$ for
$\Var_X(I)$ and denote by $\mu_i$ the multiplicity of $\zeta_i$. The
number of solutions, counting these multiplicities, is
$\DD := \mu_1 + \cdots + \mu_\D \geq \D$.  Recall that,
$I^\sigma = \I(U_\sigma) \subset \C[U_\sigma]$, where $(.)^\sigma$ is
the \emph{dehomogenization map}.

We can write $V_X(I)$ as a union of affine closed schemes
$Y_1,\dots,Y_\D$, where ${Y_i}$ corresponds to the point $\zeta_i$ in
$V_X(I)$.
For each $i = 1, \ldots, \D$, we denote by $\QQ_i$
the ideal sheaf of $Y_i$ on $X$, that is, the ideal sheaf $\QQ_i$ such
that $\OO_{Y_i} \simeq \OO_X / \QQ_i$. Note that, for any
$\sigma \in \Sigma$, $\QQ_i(U_\sigma) = H^0(U_\sigma,\QQ_i)$ is the
primary ideal of $\C[U_\sigma]$ corresponding to the point $\zeta_i$
in $V_{U_\sigma}(I^\sigma)$ if $\zeta_i \in U_\sigma$, and
$\QQ_i(U_\sigma) = \C[U_\sigma]$ otherwise.
These primary ideals relate to the primary ideals $Q_i$ in
\eqref{eq:algdecomp} as follows,
$$ \C[U_\sigma]/I^\sigma = \bigoplus_{\z_i \in U_\sigma} \C[U_\sigma]/\QQ_i(U_\sigma) \simeq \bigoplus_{\z_i \in U_\sigma} \C[y_1,\ldots,y_\ell]/Q_i.$$
Additionally,
$ (\C[U_\sigma]/\QQ_i(U_\sigma))^\vee \simeq \QQ_i(U_\sigma)^\perp
\simeq Q_i^\perp =\eval_{\z_i} \circ V_i, $  with
$$\QQ_i(U_\sigma)^\perp := \{ v \in \C[U_\sigma]^\vee ~|~ v(f^\sigma) = 0, \text{ for all } f^\sigma \in \QQ_i(U_\sigma) \}$$
and with $V_i \subset \Diff$ a closed subspace of differential
operators. This allows us to write elements of $\QQ_i(U_\sigma)^\perp$
as $\eval_{\z_i} \circ \partial$, with $\partial \in V_i$. Note that
$\eval_{\z_i}$ depends on the cone $\sigma$, even though we do not
make it explicit in the notation.
For an element $\alpha \in \Cl(X)$ we denote by $I_\alpha^\perp$ the
vector space
$$I_\alpha^\perp := \{ v \in S_\alpha^\vee ~|~ v(f) = 0, \text{ for all } f \in I_\alpha\} \simeq (S/I)_\alpha^\vee.$$

Our first objective is to map the elements in $\QQ_i(U_\sigma)^\perp$
to non-zero elements in $I_\alpha^\perp$. When $\alpha \in \Pic(X)$,
we can do so by considering the dehomogenization of the elements in
$I_\alpha$. When $\alpha \in \Cl(X)$, as $(S_\alpha)^\sigma$ might be
zero, we need to lift the elements of $I_\alpha$ to a higher degree
such that their image under the dehomogenization morphism is not
trivially zero.

\begin{lemma}
  \label{lem:funcs}
  Consider $\sigma \in \Sigma(n)$ such that $\zeta_i \in U_\sigma$
  and take any $\alpha,\gamma \in \Cl(X)$ such that
  $(S_\gamma)^\sigma \neq 0$. For each $g \in S_{\gamma - \alpha}$.
  and
  $\eval_{\zeta_i} \circ \partial \in \QQ_i(U_\sigma)^\perp$, we have
  $(f \mapsto (\eval_{\zeta_i} \circ \partial) \left( (g \, f)^\sigma
  \right) ) \in I_{\alpha}^\perp$.
\end{lemma}

\begin{proof}
  Suppose $\eval_{\zeta_i} \circ \partial \in
  \QQ_i(U_\sigma)^\perp$. For any element $f \in I_\alpha$,
  $g \, f \in I_\gamma$ and so
  $(g \, f)^\sigma \in I^\sigma \subset
  \QQ_i(U_\sigma)$. Hence \linebreak
  $\left(f \mapsto (\eval_{\zeta_i} \circ \partial)((f \, g)^\sigma)\right) \in
  I_\alpha^\perp$.
\end{proof}

Our next goal is to show that for polynomials whose degree $\alpha$ is in the regularity of
$I$, ideal membership can be tested, after dehomogenizing, in the
affine rings $\C[U_\sigma]$.
As we explained before, we need to lift the elements in $S_\alpha$ by
multiplying them with some $g \in S_{\gamma - \alpha}$ such that
$(S_\gamma)^\sigma \neq 0$. We cannot choose any
$g \in S_\gamma$ to do so, as it might happen that $g \, f \in I$ but
$f \not\in I$.
For any $V_X(I)$-basepoint free $\alpha$, we can avoid this problem. 

\begin{proposition}
  \label{thm:toric0HilbertNull}
  
  Consider a zero dimensional homogeneous ideal $I \subset S$ and a
  $V_X(I)$-basepoint free degree $\alpha \in \Cl(X)$. For any element
  $h \in S_\alpha$ such that $h$ does not vanish at any point of
  $\Var_X(I)$ and any full-dimensional cone $\sigma \in \Sigma(n)$,
  there exists $\gamma \in \Cl(X)$ and
  $g_\sigma \in S_{\gamma - \alpha}$ such that
  $(g_\sigma \, h)^\sigma \neq 0$ and
  $(g_\sigma \, h)^\sigma - 1 \in I^\sigma$. Moreover, there is
  $\ell \in \N$ such that
  $(x^{\hat{\sigma}})^\ell \, (g_\sigma \, h -
  x^{\hat{\sigma},\gamma}) \in I$.
\end{proposition}

\begin{proof}
  Let $U_\sigma$ be the affine open set of $X$ associated to the
  full-dimensional cone $\sigma \in \Sigma(n)$. Consider the ideal
  $\ideal{h} + I$. As the
  dehomogenization map $(.)^\sigma$ is linear, we have
  $(\ideal{h} + I)^\sigma = \ideal{h}^\sigma + I^\sigma = (( \ideal{h} + I )_{x^{\hat{\sigma}}})_0$.
  Since we assumed
  that $h$ does not vanish at any of the points of $\Var_X(I)$, by
  Hilbert's weak Nullstellensatz, we have $1 \in \ideal{h}^\sigma +
  I^\sigma$.
  If $\Var_X(I) \cap U_\sigma \neq \emptyset$, then $1 \not\in I^\sigma$,
  and by Lemma \ref{lem:clearDenom} there is $\gamma \in \Cl(X)$ and a non-zero $g_\sigma \in S_{\gamma - \alpha}$ such that
  $(g_\sigma \, h)^\sigma - 1 \in I^\sigma$. It is clear that
  $(g_\sigma \, h)^\sigma \neq 0$.
  If $1 \in I^\sigma$, we can take any
  $g_\sigma \in S_{\gamma - \alpha}$ such that
  $(g_\sigma \, h)^\sigma \neq 0$. Note that $(g_\sigma \, h)^\sigma - 1 = (g_\sigma h - x^{\hat{\sigma},\gamma})^\sigma \in I^\sigma$, and hence by Lemma \ref{lem:clearDenom}, $ (x^{\hat{\sigma}})^\ell (g_\sigma h - x^{\hat{\sigma},\gamma}) \in I$ for some $\ell \in \N$. 
\end{proof}

In what follows, we consider a $V_X(I)$-basepoint free
$\alpha \in \Cl(X)$ and $h_\alpha \in S_\alpha$ such that $\Var_X(h_\alpha) \cap \Var_X(I) = \emptyset$. By
\Cref{thm:toric0HilbertNull}, for each $\sigma \in \Sigma(n)$, there
are $\gamma \in \Cl(X)$ and $g_{\sigma} \in S_{\gamma - \alpha}$ such
that $(h_\alpha \, g_\sigma)^\sigma - 1 \in I^\sigma$ and
$(h_\alpha \, g_\sigma)^\sigma \neq 0$. We define the map
\begin{align}
  \label{eq:eta}
  \eta_{\alpha, \sigma} : S_\alpha \rightarrow \left (S_{x^{\hat{\sigma}}} \right )_0 \simeq \C[U_\sigma] \textup{ such that } f \mapsto \eta_{\alpha, \sigma}(f) := (g_\sigma \, f)^\sigma.
\end{align}

\begin{lemma} \label{lem:compat}

  Consider $\alpha \in \Reg(I)$. Then, $f \in I_\alpha$ if and only if
  $\eta_{\alpha,\sigma}(f) \in I^\sigma$, for
  all $\sigma \in \Sigma(n)$.
\end{lemma}
\begin{proof}
We assume $f \in I_\alpha \setminus \{0\}$, as the case $f = 0$ is trivial. It is clear that if $f \in I_\alpha$, then
  $\eta_{\alpha,\sigma}(f) \in I^\sigma$. Conversely, suppose that
  $\eta_{\alpha,\sigma}(f) = (g_{\sigma} \cdot f)^{\sigma} \in
  I^\sigma$, for all $\sigma \in \Sigma(n)$.
  As we observed in the paragraph above \eqref{eq:dehomideal}, as $f$
  is not zero and $(g_{\sigma} \cdot h_\alpha)^{\sigma} \neq 0$, then
  $(g_{\sigma} \cdot f)^{\sigma} \neq 0$.
  By \Cref{lem:clearDenom}, for each $\sigma \in \Sigma(n)$, there is
  $ \ell \in \N$ such that
  $(x^{\hat{\sigma}})^\ell \, g_{\sigma} \, f \in I$. Moreover, by
  \Cref{thm:toric0HilbertNull}, there is $i \in \N$ such that
  $(x^{\hat{\sigma}})^i \, (g_{\sigma} \, h_\alpha - x^{\hat{\sigma},\gamma}) \, f
  \in I$. As $x^{\hat{\sigma},\gamma}$ divides $(x^{\hat{\sigma}})^j$
  for big enough $j \in \N$, $(x^{\hat{\sigma}})^{i+\ell+j} \, f \in I$.
  Since $B = \ideal{x^{\hat{\sigma}} ~|~ \sigma \in \Sigma(n) }$, we
  have that $f \in (I:B^\infty) = J$. As $f \in S_\alpha$ and
  $\alpha \in \Reg(I)$, we conclude $f \in J_\alpha = I_\alpha$.
\end{proof}

In our setting, we can solve the ideal membership problem by using the
operators from \Cref{subsec:affinefat}.

\begin{corollary} \label{cor:inv}
  Fix $\alpha \in \Reg(I)$ and for each $\sigma \in \Sigma$, consider
  the map $\eta_{\alpha,\sigma}$ from \eqref{eq:eta}.
  For each $\zeta_i \in \Var_X(I)$, consider $\sigma_i \in \Sigma(n)$
  such that $\zeta_i \in U_{\sigma_i}$ and let
  $\{ \eval_{\zeta_i} \circ \partial^{i1}, \ldots, \eval_{\zeta_i}
  \circ \partial^{i \mu_i} \}$ be a basis for
  $\QQ_i(U_{\sigma_i})^\perp$.
  For $f \in S_\alpha$, we have that
  $f \in I_\alpha$ if and only if
\begin{equation} \label{eq:diffconditions}
(\eval_{\zeta_i} \circ ~ \partial^{ij} \circ \eta_{\alpha,\sigma_i})(f) = 0, \quad i = 1, \ldots, \D, ~j = 1, \ldots, \mu_i.
\end{equation}
\end{corollary}
\begin{proof}
As for Lemma \ref{lem:compat}, one implication is obvious. Note that the statement is independent of the choice of $\sigma_i \in \Sigma(n)$ such that $\zeta_i \in U_{\sigma_i}$ by the fact that the $\QQ_i$ are coherent sheaves. Suppose that $f$ satisfies \eqref{eq:diffconditions}. We have that $(g_{\sigma_i} \, f)^{\sigma_i} \in I^{\sigma_i}$ if and only if $(g_{\sigma_i} \, f)^{\sigma_i} \in \QQ_i(U_{\sigma_i})$ for all $i$ such that $\zeta_i \in U_{\sigma_i}$, which is true for all ${\sigma_i} \in \Sigma$ by assumption. Therefore, by Lemma \ref{lem:compat}, we have $f \in I_\alpha$.   
\end{proof}
In what follows, for each $\zeta_i \in \Var_X(I)$, we fix
$\sigma_i \in \Sigma(n)$ such that $\zeta_i \in U_{\sigma_i}$.
As in \Cref{subsec:affinefat}, let
$\{\partial^{i1}, \ldots, \partial^{i \mu_i} \}$ be a consistently
ordered basis for $V_i$ with
$\QQ_i(U_{\sigma_i})^\perp \simeq \eval_{\zeta_i} \circ V_i$.
Additionally, we fix a regularity pair (see Definition \ref{def:regintro})
$(\alpha,\alpha_0) \in \Cl(X)^2$ and consider, for each $\sigma_i$,
a triplet of homomorphisms $\eta_{\alpha_0,\sigma_i}$, $\eta_{\alpha,\sigma_i}$,
and $\eta_{\alpha + \alpha_0,\sigma_i}$ as in \eqref{eq:eta} such that
for every $f_\alpha\in S_{\alpha}$ and $f_{\alpha_0} \in S_{\alpha_0}$ we have
\begin{align}\label{eq:defFactorEta}
\eta_{\alpha_0,\sigma_i}(f_\alpha) \,
\eta_{\alpha,\sigma_i}(f_{\alpha_0}) = \eta_{\alpha +
  \alpha_0,\sigma_i}(f_{\alpha} \, f_{\alpha_0}).
\end{align}
This triplet exists because for every pair 
$h_\alpha \in S_\alpha$ and $h_{\alpha_0} \in S_{\alpha_0}$ such that
$h_\alpha$ and $h_{\alpha_0}$ do not vanish anywhere on $\Var_X(I)$, the polynomial
$h_\alpha \, h_{\alpha_0} \in S_{\alpha+\alpha_0}$ does not vanish anywhere on
$\Var_X(I)$ either.

\begin{remark}
  \label{rmk:PicCase}
  Whenever $\alpha \in \Picbpf(X)$, as $(S_\alpha)^\sigma \neq 0$ for
  every $\sigma \in \Sigma(n)$, we can simplify $\eta_{\alpha,\sigma}$
  and define it as $\eta_{\alpha,\sigma} = (\cdot)^\sigma$ for each
  $\sigma \in \Sigma(n)$. In this case, the proof of \Cref{lem:compat}
  follows in a similar way. 
  In particular, if
  $(\alpha,\alpha_0) \in \Picbpf(X)^2$, we can set
  $\eta_{\beta,\sigma} = (f \mapsto f^\sigma)$, for each
  $\beta \in \{\alpha,\alpha_0,\alpha+\alpha_0\}$.
\end{remark}

For $\beta \in \{\alpha,\alpha+\alpha_0\}$, we introduce the notation
$v_{ij,\beta} := (\eval_{\zeta_i} \circ \partial^{ij} \circ \eta_{\beta,\sigma_i})
\in (S/I)_\beta^\vee$.
Recall that, as $\beta \in \Reg(I)$, $\dim_\C((S/I)_\beta) = {\DD}$.
We define the map $\psi_\beta : (S/I)_\beta \rightarrow \C^{\DD}$ by
$$\psi_\beta(f + I_\beta) = (v_{ij,\beta}(f + I_\beta) ~|~ i = 1, \ldots, \D, j = 1, \ldots, \mu_i).$$
By Corollary \ref{cor:inv}, as $\beta \in \Reg(I)$, the map $\psi_\beta$ is invertible. 
For each $g \in S_{\alpha_0}$, we define
  the \emph{multiplication map} representing `multiplication with $g$' as
  \begin{align}\label{eq:multMap}
 M_g : (S/I)_\alpha \rightarrow (S/I)_{\alpha + \alpha_0} \quad \textup{ such that } \quad  M_g(f + I_\alpha) = fg + I_{\alpha + \alpha_0}.
  \end{align}

\begin{lemma} \label{lem:invertible}
  Let $(\alpha, \alpha_0) \in \Cl(X)^2$ be a regularity pair.
  Consider $h_0 \in S_{\alpha_0}$ such that
  $\Var_X(I) \cap \Var_X(h_0) = \emptyset$. Then $M_{h_0}$ is invertible.
\end{lemma}
\begin{proof}
Note that 
$ v_{ij,\alpha+\alpha_0}(h_0 f + I_{\alpha+\alpha_0}) = (\eval_{\zeta_i} \circ \partial^{ij} \circ \eta_{\alpha + \alpha_0, \sigma_i})(h_0 f)$
and
$\eta_{\alpha + \alpha_0, \sigma_i}(h_0 f) =
\eta_{\alpha_0,\sigma_i}(h_0) \, \eta_{\alpha,\sigma_i}(f)$, by
\eqref{eq:defFactorEta}. By Leibniz' rule we have
$\partial^{ij}(\eta_{\alpha_0,\sigma_i}(h_0) \,
\eta_{\alpha,\sigma_i}(f)) = \sum_{b \in \N^\ell}
\partial_b(\eta_{\alpha_0,\sigma_i}(h_0)) \cdot s_b(\partial^{ij})(
\eta_{\alpha,\sigma_i}(f))$.
Using consistent ordering of the $\partial^{ij}$, it is an easy exercise to show that $\psi_{\alpha + \alpha_0} \circ M_{h_0} = \phi \circ \psi_\alpha$ where $\phi$ is represented by an invertible lower triangular matrix. The lemma follows, since $\alpha, \alpha+ \alpha_0 \in \Reg(I)$, implies that also $\psi_{\alpha}$ and $\psi_{\alpha+ \alpha_0}$ are invertible.
\end{proof}

Our next theorem characterizes the eigenvalues of the multiplication
maps in terms of evaluations of rational functions on the solutions of
the system.

\begin{theorem}[Toric eigenvalue theorem] \label{thm:toriceval} Let
  $(\alpha, \alpha_0) \in \Cl(X)^2$ be a regularity pair. For any
  $g \in S_{\alpha_0}$ and $h_0 \in S_{\alpha_0}$ such that
  $\Var_X(I) \cap \Var_X(h_0) = \emptyset$, consider the linear map
  $M_g \circ M_{h_0}^{-1} : (S/I)_{\alpha + \alpha_0} \rightarrow
  (S/I)_{\alpha+\alpha_0}$. For each $\zeta_i$, we consider
  $\frac{g}{h_0}(\zeta_i) :=
  \frac{\eval_{\zeta_i}(\eta_{\alpha_0,\sigma_i}(g))}{\eval_{\zeta_i}(\eta_{\alpha_0,\sigma_i}(h_0))}$. We have that 
$$\det(\lambda \id_{(S/I)_{\alpha + \alpha_0}} - M_g \circ M_{h_0}^{-1} ) = \prod_{i=1}^\D \left (\lambda - \frac{g}{h_0}(\zeta_i) \right )^{\mu_i}.$$
\end{theorem}

\begin{remark}\label{rem:pointsinU}
  Observe that the $\frac{g}{h_0}(\zeta_i)$ in the previous equation
  is independent of the choice of the cone $\sigma_i$ and
  $\eta_{\alpha_0,\sigma_i}$ that we associated to $\zeta_i$.
  When $\zeta_i$ belongs to the simplicial part of $X$,
  we can define evaluation of $\frac{f}{g}$ at $\z_i$ as
  $\frac{f(z_i)}{g(z_i)}$ for any $z_i \in \pi^{-1}(\z_i)$. This is
  well-defined because the evaluation of $\frac{f}{g}$ is invariant
  under the action of $G$ and $\pi^{-1}(\z_i)$ consists of exactly one
  $G$-orbit.
  When $\z_i$ does not belong to the simplicial part, this evaluation
  can be defined as $\frac{f(z_i)}{g(z_i)}$ for $z_i$ in the unique
  closed $G$-orbit contained in $\pi^{-1}(\z_i)$, see
  \cite[Ch.~5]{cox_toric_2011}.
\end{remark}

\begin{proof}
  Our strategy is to prove that there exist linear maps
  $L_{h_0}$ and $L_{g} : \C^{\DD} \rightarrow \C^{\DD}$
  such that $L_{h_0} \circ \psi_{\alpha + \alpha_0} \circ M_g = L_g \circ \psi_{\alpha + \alpha_0} \circ M_{h_0}$, the map $L_{h_0}$ is invertible and
\begin{equation} \label{eq:detcond}
\det(\lambda \id_{(S/I)_{\alpha + \alpha_0}} - L_{h_0}^{-1} \circ L_g ) = \prod_{i=1}^\D \left (\lambda - \frac{g}{h_0}(\zeta_i) \right )^{\mu_i}.
\end{equation}
For $\sigma \in \Sigma(n)$, let
$\widetilde{h_0}^\sigma = \eta_{\alpha_0,\sigma}(h_0), \widetilde{g}^\sigma =
\eta_{\alpha_0,\sigma}(g)$ and for any $f \in S_\alpha$, let
$\widetilde{f}^\sigma = \eta_{\alpha,\sigma}(f)$. We have that
$$v_{ij,\alpha+\alpha_0}(gf) = (\eval_{\zeta_i} \circ \partial^{ij} \circ \eta_{\alpha+\alpha_0,\sigma_i})(gf) = (\eval_{\zeta_i} \circ \partial^{ij})(\widetilde{g}^{\sigma_i} \, \widetilde{f}^{\sigma_i}).$$
Applying Leibniz' rule we find 
\begin{align*}
\partial^{ij}(\widetilde{h_0}^{\sigma_i} \widetilde{g}^{\sigma_i} \widetilde{f}^{\sigma_i}) &= \sum_{b \in \N^\ell} \partial_{b}(\widetilde{h_0}^{\sigma_i}) \cdot s_{b}(\partial^{ij})(\widetilde{g}^{\sigma_i} \widetilde{f}^{\sigma_i}) = \sum_{b \in \N^\ell} \partial_b(\widetilde{g}^{\sigma_i}) \cdot s_b(\partial^{ij})(\widetilde{h_0}^{\sigma_i}\widetilde{f}^{\sigma_i}).
\end{align*}
Composing with $\eval_{\zeta_i}$, by consistent ordering of the $\partial^{ij}$, as in \eqref{eq:mateq1} we get 
{
\begin{align*}
\underbrace{\setlength\arraycolsep{.5pt} \begin{bmatrix}
\widetilde{h_0}^{\sigma_i}(\zeta_i) \\ c_{i2}^{(1)} & \widetilde{h_0}^{\sigma_i}(\zeta_i) \\ \vdots & & \ddots \\ 
c_{i \mu_i}^{(1)} & c_{i \mu_i}^{(2)} & \hdots & \widetilde{h_0}^{\sigma_i}(\zeta_i)
\end{bmatrix}}_{L_{i,h_0}}
                                                 \!\!
&
\begin{bmatrix}
(\eval_{\zeta_i}\circ \partial^{i1}) (\widetilde{g}^{\sigma_i} \widetilde{f}^{\sigma_i}) \\
(\eval_{\zeta_i}\circ \partial^{i2}) (\widetilde{g}^{\sigma_i} \widetilde{f}^{\sigma_i}) \\
\vdots \\
(\eval_{\zeta_i}\circ \partial^{i \mu_i}) (\widetilde{g}^{\sigma_i} \widetilde{f}^{\sigma_i})
\end{bmatrix}
                                                 \!\!
  =
                                                 \!\!
  \underbrace{\setlength\arraycolsep{.5pt} \begin{bmatrix}
\widetilde{g}^{\sigma_i}(\zeta_i) \\ d_{i2}^{(1)} & \widetilde{g}^{\sigma_i}(\zeta_i) \\ \vdots & & \ddots \\ 
d_{i \mu_i}^{(1)} & d_{i \mu_i}^{(2)} & \hdots & \widetilde{g}^{\sigma_i}(\zeta_i)
\end{bmatrix}}_{L_{i,g}}
                      \!\!
\begin{bmatrix}
(\eval_{\zeta_i}\circ \partial^{i1}) (\widetilde{h_0}^{\sigma_i} \widetilde{f}^{\sigma_i}) \\
(\eval_{\zeta_i}\circ \partial^{i2}) (\widetilde{h_0}^{\sigma_i} \widetilde{f}^{\sigma_i}) \\
\vdots \\
(\eval_{\zeta_i}\circ \partial^{i \mu_i}) (\widetilde{h_0}^{\sigma_i} \widetilde{f}^{\sigma_i})
\end{bmatrix}
\end{align*}}
for some complex coefficients $c_{ij}^{(k)}, d_{ij}^{(k)}$.  Recall
that
$\widetilde{g}^{\sigma_i}
\widetilde{f}^{\sigma_i} = \eta_{\alpha+\alpha_0}(g \, f)$ and
$\widetilde{h_0}^{\sigma_i}
\widetilde{f}^{\sigma_i} = \eta_{\alpha+\alpha_0}(h_0 \, f)$.
Putting all the equations together for $i = 1, \ldots, \D$, we get 
\begin{equation} \label{eq:mateq}
\begin{bmatrix}
L_{1,h_0} \\
& L_{2,h_0} \\ 
&& \ddots \\ 
&&& L_{\D,h_0} 
\end{bmatrix}
\circ \psi_{\alpha+\alpha_0} \circ M_g = 
\begin{bmatrix}
L_{1,g} \\
& L_{2,g} \\ 
&& \ddots \\ 
&&& L_{\D,g} 
\end{bmatrix}
\circ \psi_{\alpha+\alpha_0} \circ M_{h_0},
\end{equation}
which is the desired relation $L_{h_0} \circ ~ \psi_{\alpha + \alpha_0} \circ M_g = L_g \circ ~ \psi_{\alpha + \alpha_0} \circ M_{h_0}$.
Indeed, by construction, $\widetilde{h_0}^{\sigma_i}(\zeta_i) \neq 0, \forall i$, so $L_{h_0}$ is invertible and \eqref{eq:detcond} is satisfied.
\end{proof}

\begin{example}[{Cont. Examples 
  \ref{ex:doublepillow1}-\ref{ex:doublepillow1b}}] \label{ex:doublepillow2} Recall that
  $\alpha_1 = \alpha_2 = [\sum_{i=1}^4 D_i]$. As we will see
  (\Cref{thm:regCartier}), $(\alpha,\alpha_0) \in \Pic(X)$ is a
  regularity pair for $I$. By \Cref{rmk:PicCase}, the maps
  $\eta_{\alpha},\eta_{\alpha_0},\eta_{\alpha + \alpha_0}$ coincide with dehomogenization. One can check that in the bases
\begin{align*}
\B_\alpha &= \{x_3^4x_4^4 + I_\alpha, x_1x_2x_3^3x_4^3 + I_\alpha ,x_1x_2^3x_3^3x_4 + I_\alpha ,x_1^4x_2^4 + I_\alpha \} \\
\B_{\alpha + \alpha_0} &=\{x_2^2x_3^6x_4^4 + I_{\alpha + \alpha_0}, x_1x_2^3x_3^5x_4^3 + I_{\alpha + \alpha_0}, x_1x_2^5x_3^5x_4 + I_{\alpha + \alpha_0} , x_1^4x_2^6x_3^2 + I_{\alpha + \alpha_0} \}
\end{align*}
of $(S/I)_\alpha$ and $(S/I)_{\alpha + \alpha_0}$ respectively, multiplication with $x_2^2x_3^2, x_1x_2x_3x_4 \in S_{\alpha_0}$ looks like this: 
$$ M_{x_2^2x_3^2} = \begin{bmatrix}
1&0&0&0\\0&1&0&0\\0&0&1&0\\0&0&0&1
\end{bmatrix}, \qquad M_{x_1x_2x_3x_4} = \begin{bmatrix}
0&0&0&0\\ 
1&0&0&1\\
0&0&0&-1\\
0&0&0&0\\
\end{bmatrix}.$$
The solution $\z_1 = \pi(z_1) \in U_{\sigma_1}$ has local coordinates $(y_1,y_2,y_3) = (1,0,0)$, and a consistently ordered basis for $Q_1^\perp$ is 
$$ \{ \eval_{\z_1} \circ \partial_{(0,0,0)}, \eval_{\z_1} \circ \partial_{(0,1,0)}  \} = \{ \eval_{\z_1}, \eval_{\z_1} \circ \frac{\partial}{\partial y_2} \}.$$ 
composing these functionals with dehomogenization of degree $\alpha + \alpha_0$ and representing them in the basis $\B_{\alpha + \alpha_0}$ we get 
$$\begin{bmatrix}
v_{11,\alpha + \alpha_0} \\
v_{12,\alpha + \alpha_0} 
\end{bmatrix} = \begin{bmatrix}
\eval_{\z_1} \circ \eta_{\alpha+\alpha_0, \sigma_1} \\ 
\eval_{\z_1} \circ \frac{\partial}{\partial y_2} \circ \eta_{\alpha+\alpha_0, \sigma_1} 
\end{bmatrix}
= \begin{bmatrix}
1&0&0&0\\
0&1&1&0
\end{bmatrix},$$
which follows from $\eta_{\alpha + \alpha_0} (\B_{\alpha + \alpha_0}) = \{y_1,y_1y_2,y_1^2y_2,y_1y_2^4 \}$. For any $g \in S_{\alpha_0}$, the lower triangular matrix $L_{1,g}$ is given by 
$$ L_{1,g} = \def\arraystretch{1.5} \begin{bmatrix}
g^{\sigma_1} (\z_1) & 0\\ 
\frac{\partial g^{\sigma_1}}{\partial y_2}(\z_1) & g^{\sigma_1} (\z_1)
\end{bmatrix}, \quad \text{which gives} \quad L_{1,g} = \begin{bmatrix}
0 & 0 \\ 1 & 0
\end{bmatrix}, \qquad L_{1,h_0} = \begin{bmatrix}
1 & 0 \\ 0 & 1
\end{bmatrix}$$
for $g = x_1x_2x_3x_4, h_0 = x_2^2x_3^2$.
This follows from $g^{\sigma_1} = y_2$, $h_0^{\sigma_1} = y_1$. For the rows of \eqref{eq:mateq} corresponding to $\z_1$ we get 
$$ \begin{bmatrix}
1&0\\0&1
\end{bmatrix} \begin{bmatrix}
1&0&0&0\\0&1&1&0
\end{bmatrix} M_{x_1x_2x_3x_4} = \begin{bmatrix}
0&0\\1&0
\end{bmatrix} \begin{bmatrix}
1&0&0&0\\0&1&1&0
\end{bmatrix} M_{x_2^2x_3^2} .$$
In order to complete this equation with the rows corresponding to $\z_2$, one has to work in the chart corresponding to either one of the orange or the yellow cone in Figure \ref{fig:doublepillow}.
\end{example}

\subsection[Computing coordinates of VarX(I)]{Computing coordinates of $\Var_X(I)$} \label{subsec:algorithm}
We now use the results from the previous sections to design a numerical algorithm for computing homogeneous coordinates of the points in $\Var_X(I)$. We sketch the steps of the algorithm. More details on the numerical aspects of the strategy can be found in \cite[Sec. 5.5.4]{telen2020thesis}. Let $I = \ideal{f_1, \ldots, f_s} \subset S$ be such that $\Var_X(I) = \{\z_1, \ldots, \z_\D \}$ where $\z_i$ has multiplicity $\mu_i$ and set $\DD := \mu_1 + \cdots + \mu_\D$. We write $\alpha_i = \deg(f_i) \in \Cl(X)_+$. For a regularity pair $(\alpha, \alpha_0) \in \Cl(X)^2$ and some $h_0 \in S_{\alpha_0}$ for which $\Var_X(I) \cap \Var_X(h_0) = \emptyset$, we fix a basis for $(S/I)_\alpha$ and compute the matrices 
$$ M_{x^{b}/h_0} = M_{x_b} \circ M_{h_0}^{-1}, \quad \text{ for all } x^b \in S_{\alpha_0}.$$
This can be done as follows \cite[Prop.~5.5.5]{telen2019numerical}. Consider the map $$\Res : S_{\alpha + \alpha_0 - \alpha_i} \times \cdots \times S_{\alpha+ \alpha_0 - \alpha_s} \rightarrow S_{\alpha + \alpha_0} \quad \text{given by } \quad (q_1, \ldots, q_s) \mapsto q_1f_1 + \cdots + q_sf_s.$$
This map has the property that $\im \, \Res = I_{\alpha + \alpha_0}$, and hence a cokernel map $N: S_{\alpha + \alpha_0} \rightarrow \C^{\DD}$ satisfies $\ker N = \im \, \Res = I_{\alpha + \alpha_0}$. Such a cokernel map can be computed, for instance, using the SVD. We define the map $N_{h_0} : S_\alpha \rightarrow \C^{\DD}$ by setting $N_{h_0}(f) = N(h_0 f)$ and for $x^b \in S_{\alpha_0}$ we define the map $N_b : S_\alpha \rightarrow \C^{\DD}$ by $N_b(f) = N(x^b f)$. For any $\DD$-dimensional subspace $W$ of $S_\alpha$ such that $(N_{h_0})_{|W}$ is invertible, we let 
$ M_{x^b/h_0} = (N_{h_0})_{|W}^{-1} \circ (N_b)_{|W}$.
It is crucial for the numerical stability of the algorithm to choose $W$ such that $(N_{h_0})_{|W}$ is well-conditioned. This can be done, for instance, by using QR with column pivoting or SVD on the matrix $N_{h_0}$, see for instance \cite{telen2018stabilized,telen2018solving,mourrain2019truncated}. Following \cite{corless1997reordered}, we compute a reordered Schur factorization of a random $\C$-linear combination $M_{h/h_0} =  \sum_{x^b \in S_{\alpha_0}} c_b \, M_{x^b/h_0}$. Provided that the degree $\alpha_0$ is `large enough' (see below), taking this random combination helps to separate the eigenspaces corresponding to different roots, as they correspond to different eigenvalues $\frac{h}{h_0}(\z_i)$, see \cite[Sec.~7]{telen2018solving}. The aforementioned Schur factorization of $M_{h/h_0}$ gives a unitary matrix $\Ubold$ such that for each $x^b \in S_{\alpha_0}$,
$$ \Ubold M_{x^b/h_0} \Ubold^H = \begin{bmatrix}
\Delta_1^b & \times & \cdots & \times \\ 
0 & \Delta_2^b & \cdots & \times\\
0 & 0 & \ddots & \vdots \\ 
0 & 0 & \hdots & \Delta_\delta^b
\end{bmatrix} \quad \text{($\cdot^H$ denotes the Hermitian transpose)}$$
is block upper triangular with the matrices $\Delta_i^b \in \C^{\mu_i \times \mu_i}$ on the diagonal. The matrices $\Delta_i^b$ are such that they only have one eigenvalue, which is $x^b/h_0$ evaluated at $\z_i$. This eigenvalue can be computed as $ \lambda_{b,i} := \textup{Trace}(\Delta_i^b)/\mu_i$. In the reduced case, where $\mu_i = 1$ for all $i$, the eigenvalues can be read off the diagonal of $\Ubold M_{x^b/h_0} \Ubold^H$. Having computed these eigenvalues, a set of homogeneous coordinates of the points $\z_1, \ldots, \z_\D$ can be computed by solving the binomial systems of equations
$$ \{ x^b = \lambda_{b,i} ~|~ x^b \in S_{\alpha_0} \} \quad \text{for $i = 1, \ldots, \DD$},$$
provided that $\alpha_0$ is `large enough'. What `large enough' means in this context is specified in \cite[Cor. 5.5.2]{telen2020thesis}. To give some intuition, if $\z_i \in T \subset X$, it is necessary and sufficient that the lattice points in the polytope associated to $\alpha_0$ affinely span the lattice $M$.\footnote{This is to ensure that $t \mapsto (t^m)_{F^\top m + a_0 \geq 0}$ is injective on $T$, where $\alpha_0 = [\sum_{i=1}^k a_{0,i}D_i]$.} In theory, this restriction can be avoided by working with eigenvectors.

\begin{remark}[Using eigenvectors instead of eigenvalues]
Our characterization of the
  eigenvectors in the proof of \Cref{thm:toriceval} gives an alternative method to compute homogeneous coordinates of the points in $\Var_X(I)$, as in \cite{auzinger_elimination_1988}. In analogy with what is observed in \cite[Rmk. 4.3.4]{telen2020thesis}, the approach using eigenvalues gives more accurate results. However, we point out that the eigenvector approach might work for smaller degrees $\alpha_0$, i.e. degrees $\alpha_0$ not satisfying the restrictions of \cite[Cor. 5.5.2]{telen2020thesis}, see above. 
\end{remark}
 
\begin{example}[27 lines on a cubic surface]
  \label{ex:27lines}
  A classical result in intersection theory states that a general
  cubic surface in $\PP^3$ contains 27 lines, see for instance
  \cite[Sec.~6.2.1]{eisenbud20163264}. As detailed in
  \cite[Sec.~4]{panizzut2019octanomial}, these lines correspond to the
  solutions of the system defined by
\begin{align*}
\small
\f_1 &= c_0t^3+c_1t^2v+c_2tv^2+c_3v^3+c_4t^2+c_5tv+c_6v^2+c_7t+c_8v+c_9,\\
\f_2 &= c_0s^3+c_1s^2u+c_2su^2+c_3u^3+c_{10}s^2+c_{11}su+c_{12}u^2+c_{16}s+c_{17}u+c_{19},\\
\f_3 &= 3c_0st^2+2c_1stv+c_2sv^2+c_1t^2u+2c_2tuv+3c_3uv^2+2c_4st+c_5sv+c_{10}t^2\\
&+c_5tu+c_{11}tv+2c_6uv+c_{12}v^2+c_7s+c_{13}t+c_8u+c_{14}v+c_{15},\\
\f_4 &= 3c_0s^2t+c_1s^2v+2c_1stu+2c_2suv+c_2tu^2+3c_3u^2v+c_4s^2+2c_{10}st+c_5su\\
&+c_{11}sv+c_{11}tu+c_6u^2+2c_{12}uv+c_{13}s+c_{16}t+c_{14}u+c_{17}v+c_{18},
\end{align*}
for general $c_i \in \C$, see \cite[Eq. (14)]{panizzut2019octanomial}. The relations $\f_1 = \ldots =  \f_4 = 0$ on $(\C^*)^4$ extend naturally to a toric compactification $X = X_{\Sigma} \supset (\C^*)^4$, where $X_\Sigma$ is the toric variety coming from the fan $\Sigma$ that we will now describe.
For $i = 1, \ldots, 4$, let $P_i \subset \R^4$ be the Newton polytope of $\f_i$ and define the convex polytope $P := P_1 + \cdots + P_4 =\{ m \in \R^4 ~|~ F^\top m + a \geq 0 \}$, with 
$$F = \begin{bmatrix}
0 & 0 & -1 & 0 & 1 & 0\\
0 & 0 & 0  & -1& 0 & 1\\
1 & 0 & -1 & 0 & 0 & 0\\
0 & 1 & 0 & -1 & 0 & 0
\end{bmatrix} = [u_1~u_2~u_3~u_4~u_5~u_6] \quad \textup{and} \quad a =(
0,0,6,6,0,0)^\top.$$
The fan $\Sigma$ is the normal fan of $P$. It has 6 rays, whose primitive generators $u_i$ are the columns of $F$.
The toric variety $X_\Sigma$ is isomorphic to the multiprojective
space $\PP^2 \times \PP^2$ and its class group is $\Z^2$. We identify
$ [\sum_{i=1}^6 c_i D_i ] = (c_1 + c_3 + c_5, c_2 + c_4 + c_6) \in \Z^2$.
This way, $\f_1,\dots,\f_4$
correspond to homogeneous polynomials $f_1,\dots,f_4$ in the Cox ring of
$X_\Sigma$ of degrees $(0,3),(3,0),(1,2)$ and $(2,1)$ respectively. As we will show (\Cref{thm:regCartier}), we have that $((6,6),(1,1))$ is a regularity pair for $I$.

The mixed volume $\MV(P_1, P_2,P_3,P_4)$ is 45.
The toric version of the BKK theorem, see \cite[\S 5.5]{fulton1993introduction}, tells us that the maximal number of isolated solutions of $f_1 = \cdots = f_4=0$ on $X$ is $45$. However, we know from intersection theory that for generic parameter values $c_0, \ldots, c_{19}$, there are only 27 solutions in $(\C^*)^4$.
Solving a generic instance of our system using the aforementioned regularity pair in a proof of concept implementation, we find that there are in fact 45 isolated solutions on $X$ (counting multiplicities), of which 18 are on the boundary $X \setminus (\C^*)^4$.
The left part of Figure \ref{fig:mult27lines} shows the computed coordinates. 
The figure suggests clearly that there are indeed 27 solution in the torus, and 18 solutions that are on the intersection of the 3rd and 4th torus invariant prime divisors (corresponding to $u_3$ and $u_4$), which we will denote by $D_3, D_4 \subset X$. In fact, there are only 3 solutions on $D_3 \cap D_4$, each with multiplicity 6. These multiplicities become apparent when the $\Ubold$ matrix in the ordered Schur factorization of a generic linear combination of the $M_{x^{b}/h_0}$ brings the matrices $M_{x^{b}/h_0}$ into \emph{block upper triangular} instead of \emph{upper triangular} form, as described above. One of the matrices $\Ubold M_{x^{b}/h_0} \Ubold^H$ is shown in the right part of Figure \ref{fig:mult27lines}. It is clear that this is a numerical approximation of a block upper triangular matrix with three $6 \times 6$ blocks on its diagonal.
\begin{figure}[h!]
\centering
\raisebox{2cm}{\includegraphics[scale=1.7]{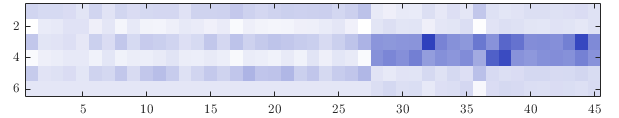}} \quad
\includegraphics[scale=1.4]{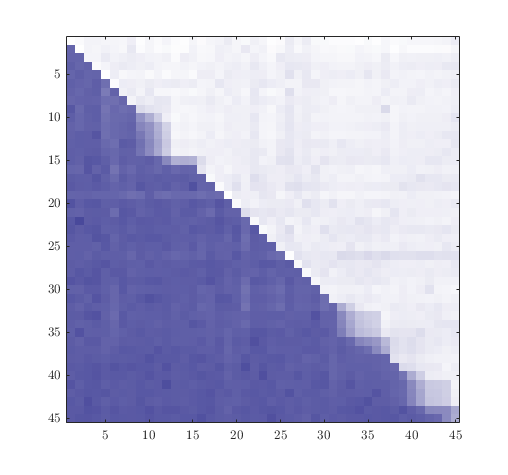}
\caption{Left: absolute value of the computed homogeneous coordinates of 45 solutions. The $i$-th row corresponds to the $i$-th torus invariant prime divisor, associated to the ray generated by $u_i$, and the $j$-th column corresponds to the $j$-th computed solution. Dark colors correspond to small absolute values. Right: absolute values of the entries of a block upper triangularized homogeneous multiplication matrix $M_{x^{b}/h_0}$ in Example \ref{ex:27lines}. }
\label{fig:mult27lines}
\end{figure}
We now explicitly compute the three solutions on the boundary by solving the \emph{face system} (see e.g. \cite{huber1995polyhedral})
corresponding to $u_3$ and $u_4$: 
\begin{align*}
\small
(\f_1)_{u_3,u_4}(s,u,t,v) &= c_0t^3+c_1t^2v+c_2tv^2+c_3v^3, \\
(\f_2)_{u_3,u_4}(s,u,t,v) &= c_0s^3+c_1s^2u+c_2su^2+c_3u^3, \\
(\f_3)_{u_3,u_4}(s,u,t,v) &= 3c_0st^2+2c_1stv+c_2sv^2+c_1t^2u+2c_2tuv+3c_3uv^2, \\
(\f_4)_{u_3,u_4}(s,u,t,v) &= 3c_0s^2t+c_1s^2v+2c_1stu+2c_2suv+c_2tu^2+3c_3u^2v.
\end{align*}
One can see from these equations that $D_3 \cap D_4 \simeq \PP^1 \times \PP^1$, with coordinates $(s:u)$ and $(t:v)$ on the first and second copy of $\PP^1$ respectively. The bidegrees of the equations are $(0,3), (3,0), (1,2), (2,1)$. We now interpret $(\f_1)_{u_3,u_4} $ as an equation on $\PP^1$ and consider its three roots $(t^*_j:v^*_j)$, $j = 1,2,3$ (for which we can write down explicit expressions) and we define $\zeta_j = ((t^*_j:v^*_j),(t^*_j:v^*_j)) \in \PP^1 \times \PP^1$. It is clear that $(\f_1)_{u_3,u_4}(\zeta_j) = (\f_2)_{u_3,u_4}(\zeta_j) = 0$. If we substitute $s = t, u = v$ in $(\f_3)_{u_3,u_4}, (\f_4)_{u_3,u_4}$ we find
$$(\f_3)_{u_3,u_4}(t,v,t,v) = (\f_4)_{u_3,u_4}(t,v,t,v) = 3 (\f_1)_{u_3,u_4}(s,u,t,v).$$
From this it is clear that also $(\f_3)_{u_3,u_4}(\zeta_j) = (\f_4)_{u_3,u_4}(\zeta_j) = 0$, $j = 1, \ldots, 3$, and we have identified the three solutions on $D_3 \cap D_4$. 
\end{example}

\begin{remark}[Other fields]
Although our solving method was designed to work in floating point arithmetic over $\C$, it is expected to generalize for eigenvalue algorithms over other fields, as in \cite{kulkarni2020solving} for 
$p$-adic numbers.
\end{remark}


\section{Regularity for zero-dimensional ideals} \label{sec:regularity}

An important property of the Cox ring is that every homogeneous ideal $I \subset S$
determines a closed subscheme of the toric variety $X$ and every
closed subscheme arises in this way, see
\cite[Prop. 6.A.6]{cox_toric_2011}. Unfortunately, this correspondence
is not one-to-one: different homogeneous ideals may define the same
closed subscheme.
\begin{example}[{\cite[Ex. 5.3.11]{cox_toric_2011}}] \label{ex:regularity}
Let $X = \PP(1,1,2)$ be the weighted projective plane with weights $1,1,2$. Its $\Z$-graded Cox ring is $\C[x,y,z]$, where $\deg(x) = \deg(y) = 1$ and $\deg(z) = 2$. The irrelevant ideal is $B = \ideal{x,y,z}$. Consider the homogeneous ideals $I_1 = \ideal{x^2,xy,y^2}$ and $I_2 = \ideal{x,y}$. One can check that both these ideals are $B$-saturated and $V_X(I_1) = V_X(I_2)$, yet $(I_1)_d \neq (I_2)_d$ for any odd degree $d$.
\end{example}
Example \ref{ex:regularity} suggests that different homogeneous ideals
defining the same subscheme of $X$ do seem to agree at certain
degrees.
The \emph{(Castelnuovo-Mumford) regularity} formalizes this intuition for $X = \PP^n$. Roughly speaking,
the regularity of an ideal is the set of degrees at which we can recover its geometric nature. For an introduction to
this subject, see e.g. \cite[Sec.~20.5]{eisenbud_commutative_2004}.

Even though the Castelnuovo-Mumford regularity is well understood when our toric variety is
$\PP^n$, in general, the notion of regularity over the Cox ring is not
uniquely defined. Equivalent definitions on $\PP^n$, i.e.\, using
local cohomology or Betti numbers, do not agree for other toric
varieties.
For simplicial toric varieties, in \cite{maclagan2003multigraded}
the regularity is defined in terms of the vanishing of certain local cohomology
modules. In contrast, in \cite{sidman_multigradedFat_2006} the (resolution) regularity of the product of
projective spaces is defined in terms of multigraded Betti numbers. Other
works, e.g. \cite{ha_multigraded_2007,sidman_multigradedCoar_2006,botbol_castelnuovo_2017},
try to unify and/or generalize these approaches.
In \cite{ha_regularity_2004,sidman_multigradedCoar_2006},
alternative definitions for fat points in multiprojective space are studied. For more general zero-dimensional
closed subschemes of toric varieties,
\cite{csahin2016multigraded} describes the regularity (in the sense of
\cite{maclagan2003multigraded}) of homogeneous ideals generated by
regular sequences over the Cox ring.
Alternatively, in \cite{telen2019numerical}, the
second author of this paper defines the regularity of an ideal defining a zero-dimensional
reduced subscheme in terms of its Hilbert function and the existence of \emph{homogeneous Lagrange polynomials}.
This definition of regularity only considers the ideal at a specific
degree, in contrast to other definitions as
\cite[Def.~4.1]{maclagan2003multigraded} where the local cohomology
modules of the ideal have to vanish at many different degrees.

In this work, we consider a notion of regularity (see Definition
\ref{def:regintro}) which generalizes \cite[Def. 4.3]{telen2019numerical}.\footnote{We warn the reader that, even in the reduced case, the ideal $J$ in \cite{telen2019numerical} is not the same as $(I:B^\infty)$ in this work, see \cite[Lem. 5.5.2]{telen2020thesis}.} In
contrast with previous approaches such as \cite{maclagan2003multigraded}, we allow degrees outside of the Picard group:
$$ \Reg(I) :=  \{ \alpha \in \Cl(X) ~|~ \dim_\C (S/I)_\alpha = \DD, I_\alpha = (I:B^\infty)_\alpha, \textup{$\alpha$ is $V_X(I)$-basepoint free} \}.$$ 

The aim of this section is to investigate properties of the regularity. In \Cref{subsec:defregularity} we consider the general zero-dimensional case. In \ref{subsec:regpic} we construct regularity pairs for complete intersections. For special systems, we improve this construction in \ref{subsec:generalizations}.

\subsection{Properties of regularity} \label{subsec:defregularity}
In this subsection we prove that, at the degrees belonging to the
regularity, the ideal contains all the information to recover the
geometric nature of the associated closed subscheme. This fact is
closely related to \Cref{lem:compat}. Additionally, we show how to
extend degrees in the regularity to regularity pairs.

\begin{theorem}
  \label{thm:regImpliesIso}
Consider $I \subset S$ such that $Y = V_X(I)$ is zero-dimensional. If $\beta \in \Reg(I)$, then $(S/I)_\beta \simeq H^0(Y,\OY)$.
\end{theorem}

\begin{proof}
  We recall the following exact sequence relating local
  cohomology to sheaf cohomology on toric varieties,
    \cite[Thm.~9.5.7]{cox_toric_2011},
    \begin{align}
      \label{eq:relSheafLocalCohom}
  0 \rightarrow H^0_B(S/I)_\beta \rightarrow (S/I)_\beta \rightarrow
  H^0(X, \widetilde{(S/I)(\beta)}) \rightarrow H^1_B(S/I)_\beta \rightarrow 0,
    \end{align}
  where $\widetilde{(S/I)(\beta)}$ denotes the coherent sheaf associated to the
  $S$-module $S/I$ shifted by $\beta$. 
  
  By definition, $H^0_B(S/I)_\beta = 0$ if and only if
  $I_\beta = (I:B^\infty)_\beta$. Hence, if $\beta \in \Reg(I)$,
  $H^0_B(S/I)_\beta = 0$ and we have an injective map
  $(S/I)_\beta \rightarrow H^0(X,\widetilde{(S/I)(\beta)})$.

We prove that for $\beta \in \Reg(I)$, $\widetilde{(S/I)(\beta)} \simeq  \widetilde{S/I}$.
  As $\beta \in \Reg(I)$, there is $h \in S_\beta$ such that $h$ does
  not vanish at any of the points of $Y_{\text{red}}$. We consider the morphism of
  sheaves $\times h: \widetilde{S/I} \rightarrow\widetilde{(S/I)(\beta)}$ given by
  multiplication by $h$ (which is a global section of $\OX(\beta)$). 
  We claim that this morphism is an isomorphism. This can be seen from the fact that it is an isomorphism on each chart of the affine covering
  $\{U_\sigma\}_{\sigma \in \Sigma(n)}$, given by the fan $\Sigma$ associated to
  $X$. Following the notation used in the proof of
  \Cref{{thm:toric0HilbertNull}}, for each $U_\sigma$, the inverse of $H^0(U_\sigma,\widetilde{S/I}) \xrightarrow{\times h}
 H^0(U_\sigma,\widetilde{(S/I)(\beta)})$ is
$$  \left ( (S/I)_{x^{\hat{\sigma}}} \right )_\beta = H^0(U_\sigma,\widetilde{(S/I)(\beta)}) \xrightarrow{\times
    \frac{g_\sigma}{x^{\hat{\sigma},\gamma}}}  \left ( (S/I)_{x^{\hat{\sigma}}} \right )_0 =  H^0(U_\sigma,\widetilde{S/I}),$$
  where $\gamma \in \Cl(X)$ is such that
  $(g_\sigma \, h)^\sigma = h \,
  \frac{g_\sigma}{x^{\hat{\sigma},\gamma}}$.

  By definition of $\Reg(I)$, $\Dim_\C((S/I)_\beta)$ is the
  number of solutions counting multiplicity, which is also
  $\Dim_\C(H^{0}(Y,\OY))$.
  As $\widetilde{S/I} = \mathrm{i}_*\OY$, where $\mathrm{i}$ is the immersion $Y \xhookrightarrow{\mathrm{i}} X$, and $H^0(X,\mathrm{i}_* \OY) \simeq H^{0}(Y,\OY)$,
  $(S/I)_\beta \rightarrow H^0(Y, \OY)$ is an injective map between equidimensional spaces.
\end{proof}

Given a degree $\alpha \in \Reg(I)$ we can extend it to a regularity
pair $(\alpha,\alpha_0)$ by checking that
$\dim_\C((S/I)_{\alpha + \alpha_0}) = \DD = \dim_\C((S/I)_{\alpha})
$. This is easier to verify than
$I_{\alpha+\alpha_0} = (I:B^\infty)_{\alpha + \alpha_0}$.

 \begin{lemma} \label{lem:satzerodiv}
Let $I \subset S$ be such that $V_X(I)$ is zero-dimensional. For any $\alpha_0 \in \Cl(X)_+$ and $h_0 \in S_{\alpha_0}$ such that $\Var_X(h_0) \cap \Var_X(I) = \emptyset$, we have that $h_0$ is not a zero divisor in $S/J$ where $J := (I:B^\infty)$.
 \end{lemma}
 \begin{proof}
   We prove this lemma by contra-positive.  Let
   $J = Q_1\cap \cdots \cap Q_\ell$ be an irredundant primary
   decomposition. Since $J$ is $B$-saturated,
   $\Var_{\C^k}(Q_i) \not \subset \Var_{\C^k}(B)$, for every
   $i = 1, \ldots, \ell$.
   Consider $h_0 \in S_{\alpha_0}$ such that the image of $h_0$ in
   $S/J$ is a zero divisor. Hence, $h_0 \not\in J$ and there is a
   $f \notin J$ such that $h_0 \, f \in J$. Therefore, there is a
   primary ideal $Q_i$ in the decomposition of $J$ such that
   $f \not \in Q_i$ but $h_0 \, f \in Q_i$, and so $h_0 \in \sqrt{Q_i}$. As
   $\Var_{\C^k}(Q_i) \not \subset \Var_{\C^k}(B)$, we have that
   $\Var_X(h_0) \cap \Var_X(I) \neq \emptyset$.
\end{proof}

\begin{theorem} \label{thm:regQPic}
  Let $I \subset S$ be a homogeneous ideal
  such that $V_X(I)$ is zero-dimensional of degree $\DD$ and let
  $\alpha \in \Reg(I)$. For each $V_X(I)$-basepoint free $\alpha_0 \in \Cl(X)_+$
   such that
  $ \dim_\C((S/I)_{\alpha + \alpha_0}) = \DD$, we have that
  $(\alpha, \alpha_0)$ is a regularity pair of $I$.
\end{theorem}

\begin{proof}
  We only need to prove that
  $J_{\alpha+\alpha_0} = I_{\alpha+\alpha_0}$, or equivalently
  $(S/J)_{\alpha+\alpha_0} = (S/I)_{\alpha+\alpha_0}$. By definition
  of $J = (I:B^\infty)$ we have $I \subset J$, which implies
  $\dim_\C((S/J)_{\alpha + \alpha_0}) \leq \dim_\C((S/I)_{\alpha +
    \alpha_0})$, so it suffices to show the opposite inequality.
  By assumption, $\dim_\C (S/I)_{\alpha+\alpha_0} = \DD$.
  By \Cref{lem:satzerodiv}, there is a non zero-divisor $h_0 \in S_{\alpha_0}$ in $S/J$
  such that the multiplication map
  $M_{h_0}: (S/J)_{\alpha} \rightarrow (S/J)_{\alpha+\alpha_0}$
  (see \eqref{eq:multMap})
  is
  injective. We conclude that
  $\dim_\C (S/J)_{\alpha+\alpha_0} \geq \dim_\C (S/J)_{\alpha} = \DD =
  \dim_\C((S/I)_{\alpha + \alpha_0})$.
\end{proof}

\subsection{Complete intersections}
\label{subsec:regpic}
We say that the subscheme $Y = V_X(I)$ is a \emph{complete
  intersection} if $I$ can be generated by $\codim_X(Y)$ many
homogeneous elements. In particular, a zero-dimensional subscheme $Y$
is a complete intersection if $Y = V_X(I)$ with
$I = \ideal{f_1, \ldots, f_n} \subset S$. The associated system of equations $f_1 = \cdots = f_n = 0$ is called \emph{square}.
In this subsection we describe the regularity $\Reg(I)$ of a homogeneous
ideal $I \subset S$
defining a zero-dimensional complete intersection subscheme of
$X$. We assume that the generators $f_1, \ldots, f_n$ of $I$ have degrees $\deg (f_i) = \alpha_i \in \Picbpf(X)$. This is the case, for instance, when these polynomials arise from homogenization of Laurent polynomials, see \Cref{subsec:hom}. Our objective is to characterize regularity pairs in order to
solve the square system $f_1 = \cdots = f_n = 0$ on $X$ using the algorithms in Section
\ref{sec:fatpoints}.

The main results of this subsection is \Cref{thm:vanishCohomExactGlobal},
in which we establish a sufficient criterion for degrees to belong to
the regularity. 
The strategy to prove this theorem is similar to the one in
\cite{gelfand_discriminants_1994} for defining the resultant, to
\cite{cox_codimension_2004} for checking its vanishing, and to
\cite{massri2016solving} for solving affine sparse systems. 
In the terminology of
\cite{zamaere_virtual_2017}, what we do is construct a virtual
resolution for $S/I$ using the Koszul complex.
The statement and proof of
\Cref{thm:vanishCohomExactGlobal} involve some notions of homological algebra and sheaf cohomology, but its concrete consequences can be formulated without this language, see \Cref{thm:regQCartier} and \Cref{cor:extendRegCI}. 

\begin{theorem}
  \label{thm:vanishCohomExactGlobal}
  Let $I = \ideal{f_1, \ldots, f_n} \subset S$ with $f_i \in S_{\alpha_i}$ such that $\alpha_i \in \Picbpf(X)$ and $V_X(I)$ is a complete intersection. Let $\beta \in \Cl(X)$ be $V_X(I)$-basepoint free. We have that $\beta \in \Reg(I)$ if
  \begin{align} \label{cond:vanishHigherCohom}
    \text{for every $p > 0$ and $\A \subset \{1,\dots,n\}$, \quad $H^p(X,\OX(\beta - \sum_{i \in \A} \alpha_i)) = 0$.}
  \end{align}
\end{theorem}

\begin{proof}
  Let $Y = V_X(I)$ and denote its structure sheaf by $\OY$. We denote
  by $\mathrm{i}_* \OY$ the push-forward of $\OY$ through the immersion
  $Y \xhookrightarrow{\mathrm{i}} X$.
  Consider the $\OX$-module map
  $\phi : \bigoplus_{i=1}^n \OX(-\alpha_i) \rightarrow \OX$ such
  that, for every open $U_\sigma \subset X$,
  $\phi{\mid_{U_\sigma}} : (g_1,\dots,g_n) \mapsto \sum g_i \cdot
  f_i^\sigma$.  Let $\Ko(f_1,\dots,f_n)_\bullet$ be the augmented
  Koszul complex of sheaves associated to $\phi$,
  {\small
    \begin{multline*}
   0 \rightarrow \OX(- \sum_{i = 1}^n \alpha_i)
  \rightarrow \dots \rightarrow
  \bigoplus_{\substack{\A \subset \{1,\dots,n\} \\ \#\A = j}} \OX(- \sum_{i \in \A} \alpha_i)
  \rightarrow \dots \rightarrow
  \bigoplus_{i = 1}^n \OX(-\alpha_i)
  \rightarrow
  \OX
   \rightarrow \mathrm{i}_* \OY \rightarrow 0.
 \end{multline*}
}

  \noindent
  By construction, $X$ is a normal toric variety, and so it is
  Cohen-Macaulay \cite[Thm.~9.2.9]{cox_toric_2011}. As $Y$ is a complete intersection and $\alpha_i \in \Picbpf(X)$,
  the complex $\Ko(f_1,\dots,f_n)_\bullet$ is exact \cite[Ex.~17.20]{eisenbud_commutative_2004}.

  We consider the sheaf $\OX(\beta)$ and the twisted complex
  $
  \Ko(f_1,\dots,f_n)_\bullet \otimes \OX(\beta)
  $, where we twist each sheaf by $\OX(\beta)$.
  We will show that this twisted complex is exact.
  Recall that 
  $\Ko(f_1,\dots,f_n)_\bullet \otimes \OX(\beta)$ is exact if for
  every closed point $\z \in X$, the induced complex of stalks
  $(\Ko(f_1,\dots,f_n)_\bullet)_\z \otimes \OX(\beta)_\z$ is exact
  \cite[Pg.~64]{hartshorne_algebraic_1977}.
  Consider a closed point $\z \in X$. By definition of the
  $\mathrm{Tor}^{\OO_{X,\z}}_i(-,-)$ functor, see e.g.
  \cite[Sec.~6.2]{eisenbud_commutative_2004},
  the complex of stalks
  $(\Ko(f_1,\dots,f_n)_\bullet)_\z \otimes \OX(\beta)_\z$ is exact if
  and only if
  $\mathrm{Tor}^{\OO_{X,\z}}_i((\mathrm{i}_* \OY)_{\z},\OX(\beta)_\z) = 0$,
  for every $i > 0$. If either $(\mathrm{i}_* \OY)_{\z}$ or
  $\OX(\beta)_\z$ is a free $\OX_\z$-module, then these $\mathrm{Tor}$ modules vanish.
  If $\z \not\in Y$, then $(\mathrm{i}_* \OY)_{\z} = 0$, so
  $\mathrm{Tor}^{\OO_{X,\z}}_i(0,\OX(\beta)_\z) = 0$.
  Consider now $\z \in Y$. By assumption, there is a global section $h$ of $\OX(\beta)$ that does not
  vanish at $\zeta$. Therefore, any $(h/g)_\z \in \OO_{X,\z}$ is invertible in $\OO_{X,\z}$. The map
  $\OX_\z \xrightarrow{\times h} \OX(\beta)_\z$ is an isomorphism, so $\OX(\beta)_\z \simeq \OX_\z$ is a free
  $\OX_\z$-module, implying that
  $\mathrm{Tor}^{\OO_{X,\z}}_i(\mathrm{i}_* \OY_\z,\OX(\beta)_\z) = 0$, for every
  $i > 0$.

  Note that the previous argument also shows that
  $\mathrm{i}_* \OY \otimes \OX(\beta) \simeq \mathrm{i}_* \OY$.
  Recall that whenever $\alpha \in \Pic(X)$,
  $\OX(\alpha) \otimes \OX(\beta) = \OX(\alpha + \beta)$, as
  $\OX(\alpha)$ is a locally free sheaf. Hence, out of
  $(\Ko(f_1,\dots,f_n)_\bullet) \otimes \OX(\beta)$, we obtain the
  following exact complex of sheaves,
  \begin{align}
    \label{eq:twistedKoszul}
  0 \rightarrow \OX(\beta - \sum_{i = 1}^n \alpha_i)
  \rightarrow \dots \rightarrow
  \bigoplus_{i = 1}^n \OX(\beta -\alpha_i)
  \rightarrow
  \OX(\beta)
   \rightarrow \mathrm{i}_* \OY \rightarrow 0.
  \end{align}

  \noindent
  As taking sheaf cohomology commutes with direct sums
  \cite[Prop.~III.2.9, Rmk. III.2.9.1]{hartshorne_algebraic_1977},
  our hypothesis implies that, for $p > 0$,
  $$
  H^p\left(X,
  \bigoplus_{\substack{\A \subset \{1,\dots,n\} \\ \# \A = j}} \OX(\beta - \sum_{i \in \A} \alpha_i)
  \right)
  \simeq
  \bigoplus_{\substack{\A \subset \{1,\dots,n\} \\ \#\A = j}}
    H^p(X,
    \OX(\beta - \sum_{i \in \A} \alpha_i)
    )
    = 0.
  $$
  Since $Y = V_X(I)$ is a complete intersection and $I$ is
  generated by $n$ elements, $Y$ is zero dimensional.
  As $Y$ is a closed subscheme of $X$, by
  \cite[Ex.~9.0.6]{cox_toric_2011}, $H^p(X,\mathrm{i}_* \OY) =
  H^p(Y,\OY)$. Moreover, as $Y$ is zero dimensional, it is an affine
  scheme, so by Serre's criterion
  \cite[Thm.~III.3.7]{hartshorne_algebraic_1977}, $H^p(Y,\OY) = 0$ for
  $p > 0$.
  Therefore, every higher cohomology in \eqref{eq:twistedKoszul}
  vanishes. By \cite[Ch. 2, Lem.~2.4]{gelfand_discriminants_1994}, taking global sections of \eqref{eq:twistedKoszul} preserves exactness. As $H^0(X,\OX(\gamma)) = S_\gamma$ for any $\gamma \in \Cl(X)$ by \cite[Prop. 5.3.7]{cox_toric_2011}, the following complex is exact:
  \begin{equation} \label{eq:koszulcox}
 0 \rightarrow S_{(\beta - \sum_{i = 1}^n \alpha_i)} \rightarrow
  \dots \rightarrow
  \bigoplus_{i = 1}^n S_{(\beta - \alpha_i)} \rightarrow
  S_\beta \rightarrow
  (S/I)_{\beta}   \rightarrow 0.
\end{equation}
Here we used
$(S/I)_{\beta} \simeq H^0(Y,\OY)$, which follows from
the fact that the image of
$\bigoplus_{i = 1}^n H^0(X, \OX(\beta -\alpha_i)) \rightarrow
H^0(X,\OX(\beta))$ is $I_\beta$. Hence,
$\dim_\C((S/I)_{\beta}) = \dim_\C(H^0(Y,\OY)) = \DD$.

It remains to show that $I_\beta = (I:B^\infty)_\beta$. This follows
from the exact sequence relating local and sheaf cohomology \eqref
{eq:relSheafLocalCohom}, together with the isomorphism
$(S/I)_{\beta} \simeq H^0(Y,\OY)$.
\end{proof}

The vanishing of these sheaf
cohomologies can be computed in terms of the combinatorics of the
associated polytopes; see \cite[Ch.~9]{cox_toric_2011} for the
classical approach or \cite[Sec.~III.3]{altmann_immaculate_2018} for a
newer and simpler one in the case of nef $\Q$-Cartier divisors.
However, for our purpose, we can avoid these computations. The following classical vanishing theorems give a formula for the cohomologies in the relevant cases.

\begin{proposition}
  \label{thm:demBatVanishing}
  Consider a degree
  $\alpha \in \Q \Picbpf(X)$ and its associated polytope $P = \{ m \in M_\R ~|~ F^\top m + a \geq  0 \}$ (see
  \cite[Prop.~4.3.8]{cox_toric_2011}).  Let $\relInt(P)$ be the
  relative interior of the polytope $P$. We have
  \setlength{\columnseprule}{0.5pt}
  \begin{multicols}{2}
    \noindent
    \textbf{Demazure vanishing} \cite[Thm~9.2.3]{cox_toric_2011},
  \begin{itemize}
  \item  $H^0(X,\OO_X(\alpha)) \simeq \bigoplus\limits_{m \in P \cap M} \C \cdot x^{F^\top m  + a}$,
  \item For every $i > 0$, $H^i(X,\OO_X(\alpha)) \simeq 0$,
  \end{itemize}
\columnbreak
  \textbf{Batyrev-Borisov vanishing} \cite[Thm~9.2.7]{cox_toric_2011},
  \begin{itemize}
  \item
    $H^{\dim(P)}(X,\OO_X(- \alpha)) \simeq \!\!\!\!\!\!\!\! \bigoplus\limits_{m \in \relInt(P) 
      \cap M} \!\!\!\!\!\!\!\! \C \cdot x^{F^\top m + a}$,
  \item For every $i \neq \dim(P)$, $H^i(X,\OO_X(- \alpha)) \simeq 0$.
  \end{itemize}
\end{multicols}
\end{proposition}

Combining the previous results, we construct regularity pairs for any complete intersection.  

\begin{theorem}
  \label{thm:regQCartier}
  Let $I = \ideal{f_1, \ldots, f_n} \subset S$ with
  $f_i \in S_{\alpha_i}$ such that $\alpha_i \in \Picbpf(X)$ and
  $V_X(I)$ is a zero-dimensional.
  For any nef $V_X(I)$-basepoint free $\alpha_0 \in \Q \Picbpf(X)$, the degree
  $\beta = \sum_{i=1}^n \alpha_i + \alpha_0$ belongs to the regularity
  $\Reg(I)$. In particular, $\sum_{i=1}^n \alpha_i \in \Reg(I)$ and $(\sum_{i=1}^n \alpha_i, \alpha_0)$ is a regularity pair for
  $I$. 
\end{theorem}
\begin{proof}
Observe that, as $\sum_{i=1}^n \alpha_i \in \Picbpf(X)$ is basepoint free and $\alpha_0$ is $V_X(I)$-basepoint free, $\beta = \sum_{i=1}^n \alpha_i + \alpha_0$ is $V_X(I)$-basepoint free.
  Hence, by Theorem \ref{thm:vanishCohomExactGlobal}, it suffices to prove that
  $H^p(X,\OX(\beta - \sum_{i \in \A} \alpha_i)) = 0$ for $p > 0$ and
  $\A \subset \{1,\dots,n\}$.
  First, note that for any $\A \subset \{1,\dots,n\}$ we have
  $$\beta - \sum_{i \in \A} \alpha_i = 
  \sum_{i \in \{1,\dots,n\} \setminus \A} \alpha_i + \alpha_0 \quad
  \in \Q \Picbpf(X).$$ By Demazure vanishing (see \Cref{thm:demBatVanishing}),
  $H^p(X,\OX(\beta - \sum_{i \in \A} \alpha_i)) = 0$ for
  $p > 0$. The rest of the proof follows from
  \Cref{thm:vanishCohomExactGlobal}.
\end{proof}

Observe that, if $\alpha_0 \in \Picbpf(X)$, then $\alpha_0$ is $V_X(I)$-basepoint free and we can simplify
\Cref{thm:regQCartier}.

\begin{corollary} \label{thm:regCartier}
  With the notation of \Cref{thm:regQCartier}, if $\alpha_0 \in \Picbpf(X)$,
  $(\sum_{i=1}^n \alpha, \alpha_0)$ is a regularity pair.
\end{corollary}

\Cref{thm:regCartier} implies
\cite[Thm.~4.2 \& 4.3]{telen2019numerical} and bounds the regularity
of $S/I$ in the sense of \cite[Def.~3.1]{maclagan2003multigraded}
(following their notation, if
$\N\bm{\mathcal{C}} \subseteq \Picbpf(X)$, $S/I$ is
$(\sum_i \alpha_i)$-regular). 

\begin{example}
  [Cont. of \Cref{ex:27lines}]

  The pair $((6,6),(1,1))$ is a regularity pair for $I$ as $(6,6)$
  corresponds to the sum of the degrees of the polynomials
  $f_1,\dots,f_4$ and $(1,1) \in \Picbpf(X_\Sigma)$.
\end{example}

\Cref{thm:regQCartier} corrects \cite[Conj. 1]{telen2019numerical}, which states that \Cref{thm:regQCartier} holds
for $\alpha_0 \in \Cl(X)_+$ with the same restrictions. As we show
in the following example, the conjecture can fail to be true when we
consider effective Cartier divisors which are not numerically
effective. 

\begin{example}[{Counter-example to \cite[Conj.~1]{telen2019numerical}}] \label{ex:counter}
  Consider the (smooth) Hirzebruch surface $X_\Sigma$ associated to the polytope $P$
  shown in \Cref{fig:hirzsuf} together with its normal fan $\Sigma$, see \cite[Ex.~3.1.16]{cox_toric_2011}.
 \begin{figure}[h]
\centering
  \definecolor{mycolor1}{rgb}{1,0,0}
\begin{tikzpicture}
\begin{axis}[%
width=8cm,
height=4cm,
scale only axis,
xmin=-0.5,
xmax=3.5,
xtick = \empty,
ymin=-1,
ymax=2,
ytick = \empty,
axis background/.style={fill=white},
axis line style={draw=none} 
]
\addplot [color=mycolor1,solid,thick, fill opacity = 0.2, fill = mycolor1,forget plot]
  table[row sep=crcr]{%
0	0\\
2	0\\
1	1\\
0	1\\
0	0\\
};
\addplot[only marks,mark=*,mark size=1.5pt,mycolor1
        ]  coordinates {
    (0,0) (1,0) (1,1) (2,0) (0,1)
};

\node at (axis cs:0.75,0.5) {$P$};
\end{axis}
\end{tikzpicture}%
  \definecolor{mycolor1}{rgb}{1,0,0}
\definecolor{mycolor2}{rgb}{0,1,0}
\definecolor{mycolor3}{rgb}{0,0,1}
\definecolor{mycolor4}{rgb}{1,1,0}
\begin{tikzpicture}[baseline = -0.5cm, scale=.9]
  \begin{axis}[%
width=1.4in,
height=1.4in,
scale only axis,
xmin=-3,
xmax=3,
ymin=-3,
ymax=3,
ticks = none, 
ticks = none,
axis background/.style={fill=white},
axis line style={draw=none} 
]

\addplot [color=black!30!white,solid,fill opacity=0.2,fill = mycolor1,forget plot]
  table[row sep=crcr]{%
0	0\\
0	4\\
4	4\\
4	0\\
0	0\\
};
\addplot [color=black!30!white,solid,fill opacity=0.2,fill = mycolor2,forget plot]
  table[row sep=crcr]{%
0	0\\
4	0\\
4	-4\\
0	-4\\
0	0\\
};
\addplot [color=black!30!white,solid,fill opacity=0.2,fill = mycolor3,forget plot]
  table[row sep=crcr]{%
0	0\\
-4	-4\\
-4	4\\
0	4\\
0	0\\
};
\addplot [color=black!30!white,solid,fill opacity=0.2,fill = mycolor4,forget plot]
  table[row sep=crcr]{%
0	0\\
0	-4\\
-4	-4\\
0	0\\
};


\draw[-latex] (axis cs:0,0) -- (axis cs:2,0) node[anchor=north] {$u_1$};
\draw[-latex] (axis cs:0,0) -- (axis cs:0,2) node[anchor=west] {$u_2$};
\draw[-latex] (axis cs:0,0) -- (axis cs:0,-2) node[anchor=west] {$u_3$};
\draw[-latex] (axis cs:0,0) -- (axis cs:-2,-2) node[anchor=east] {$u_4$};

\end{axis}
\end{tikzpicture} 
  \caption{Polytope $P$ and its normal fan $\Sigma$.
  }
\label{fig:hirzsuf}
\end{figure}
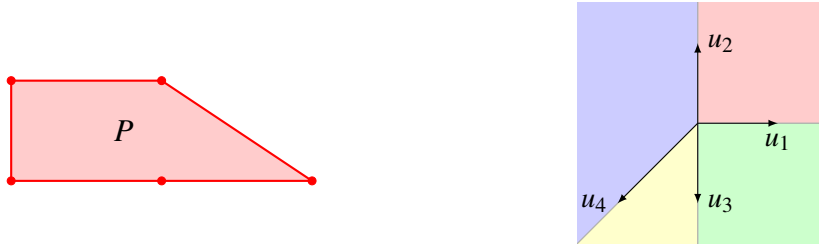
Consider the sparse polynomials $\f_1,\f_2 \in \C[M]$ from Example \ref{ex:intro} ($\varepsilon =1$) with Newton polytope $P_1=P_2 = P$:
$$
\left\{
\begin{array}{l}
   \f_1 :=    -1+t_1+t_1^2+t_2+t_1 \, t_2, \\
   \f_2 :=  -2+2 \, t_1 + 3\, t_1^2 + 4 \, t_2 + 5 \, t_1 \, t_2.
\end{array}
\right.
$$
We homogenize $\f_1,\f_2$ to obtain $f_1,f_2 \in S_{\alpha_1} = S_{\alpha_2}$ where $\alpha_1 = \alpha_2 = [D_3+2D_4]$ (here $D_i$ corresponds to $u_i$ in Figure \ref{fig:hirzsuf}).
We consider the ideal $I := \ideal{f_1,f_2} \subset S$. The closed subscheme
associated to $S/I$ is a zero-dimensional variety consisting of $3$
points in $T \subset X$. In the coordinates $(t_1,t_2)$ on $T$, we have
  $$\V_X(I) = \left\{(-2,1),\left(\frac{1}{\sqrt{2}},\frac{-3}{\sqrt{2}} +
      2\right),\left(\frac{-1}{\sqrt{2}},\frac{3}{\sqrt{2}} +
      2\right)\right\} \subset T.$$
  Consider the degrees $\alpha := [2\,(D_3+2\,D_4)]$ and
  $\alpha_0 := [2 \, D_3]$. As $\alpha = \alpha_1 + \alpha_2$, by \Cref{thm:regCartier},
  $\alpha \in \Reg(I)$ and so $\dim((S/I)_{\alpha}) = 3$.
  Since each point in $D_3$ is a basepoint of
  $S_{k \cdot \alpha_0}$, for every $k > 0$, by
  \cite[Lem.~9.2.1]{cox_toric_2011}, we have
  $\alpha_0 \not \in \Q\Picbpf(X)$. Hence, $\alpha_0$ does not satisfy
  the conditions of \Cref{thm:regQCartier}. However, $\alpha_0$ is
  $V_X(I)$-basepoint free, so $\alpha_0$ satisfies the conditions of
  \cite[Conj.~1]{telen2019numerical}.
  As $\dim_\C((S/I)_{\alpha + \alpha_0}) = 4$, we have
  $\alpha+\alpha_0 \not \in \Reg(I)$ and the conjecture fails.
\end{example}

 Combining Theorems \ref{thm:regCartier} and \ref{thm:regQPic}, we can
characterize the degrees at which \cite[Conj.~1]{telen2019numerical} holds.

\begin{corollary}
  \label{cor:extendRegCI}
  With the notation of \Cref{thm:regQCartier}, consider a $V_X(I)$-basepoint free degree $\alpha_0 \in \Cl(X)_+$. We have that $(\sum_{i=1}^n \alpha_i, \alpha_0)$ is a
  regularity pair if and only
  $\dim_\C((S/I)_{\alpha_0 + \alpha_1 + \cdots + \alpha_n}) = \DD$.
\end{corollary}

\begin{example}[Cont. Example \ref{ex:counter}] \label{ex:contcounter}
  For $\alpha := [2(D_3 + 2D_4)] \in \Picbpf(X)$ and
  $\alpha_0 := [D_3] \in \Cl(X)_+$, we check that
  $\dim_\C((S/I)_{\alpha + \alpha_0}) = 3$.
  By \Cref{cor:extendRegCI}, we
  conclude that $([ 2(D_3 + 2D_4)], [D_3])$ is a regularity pair. Note
  that $\alpha_0 \notin \Q \Picbpf(X) \supset \Picbpf(X)$.
\end{example}

As we showed in be \Cref{ex:counter}, the assumption on the Hilbert
function of $S/I$ in \Cref{cor:extendRegCI} cannot be dropped.
Nevertheless, using \Cref{lem:satzerodiv}, we can prove the following
bound on $\dim_\C((S/I)_\beta)$, which generalizes \cite[Prop.~6.7]{maclagan2003multigraded}.

\begin{lemma}
  Let $I \subset S$ be such that $V_X(I)$ is zero-dimensional. For any $V_X(I)$-basepoint free 
  $\beta \in \Cl(X)$ such that $I_\beta = (I:B^\infty)_\beta$, we have
  $\dim_\C((S/I)_{\beta}) \leq \DD$.
\end{lemma}

\begin{proof}
 By assumption,
  $H^0_B(S/I)_\beta = 0$ as $I_\beta = (I:B^\infty)_\beta$ and following the
  same argument as in \Cref{thm:regImpliesIso}, we also have
  $H^0(X, \widetilde{(S/I)(\beta)}) = H^0(Y, \OY)$. Hence, using \Cref{eq:relSheafLocalCohom}, we have an
  injective map $(S/I)_\beta \rightarrow H^0(X, \OY)$. The lemma follows as
  we defined $\DD := \dim_\C(H^0(X, \OY))$.
\end{proof}

\subsection{Further improvements} \label{subsec:generalizations}

The size of some of the matrices involved in the eigenvalue algorithm from \Cref{subsec:algorithm} is given by the dimension of
$S_{\alpha+\alpha_0}$, where $(\alpha,\alpha_0)$ is a regularity pair
for $I$.
We are therefore interested in finding regularity pairs
$(\alpha, \alpha_0)$ for which $\dim_\C S_{\alpha+\alpha_0}$ is as
small as possible.
In \Cref{thm:vanishCohomExactGlobal}, we showed how we can construct
regularity pairs by looking at the vanishing of some sheaf
cohomologies. In this subsection, we explain how this vanishing depends
on the combinatorics of the polytopes related to $X$. We illustrate
this relation by studying the regularity of unmixed, classical homogeneous, weighted homogeneous and multihomogeneous square systems.
These improvements allow us to speed up the computations
in \Cref{ex:27lines} by a factor of $25$.
Moreover, the regularity pairs that we construct lead to
matrices of roughly the same size as
the ones considered in other (affine) algebraic approaches, such as Gr\"obner
bases or sparse resultants, see for instance \cite[Sec.~8]{bender_thesis_2019} and \cite[Thm.~12]{faugere_complexity_2016}.

A straightforward consequence of Batyrev-Borisov vanishing (\Cref{thm:demBatVanishing}) is that,
whenever the polytope of $\alpha \in \Q \Picbpf(X)$ is \emph{hollow}
(no interior lattice points), the associated reflexive sheaf
$\OO_X(-\alpha)$ has no cohomologies (in the terminology of
\cite{altmann_immaculate_2018}, it is \emph{immaculate}). We can
rephrase this condition using the concept of \emph{codegree} of a
polytope. This is the
  smallest $c \in \N$ such that $c \cdot P$ contains a lattice point
  in its relative interior $\relInt(c \cdot P)$.

An important class of sparse polynomial systems consists of the so-called
\emph{unmixed sparse systems}. These are systems in which each of the
Newton polytopes is a dilation of some
lattice polytope $P$. 

\begin{theorem}
  [Unmixed sparse systems]
  \label{thm:unmixed}
  Let $\alpha_0 \in \Picbpf(X)$ be a nef Cartier divisor such that its associated lattice polytope $P$ is full dimensional. Consider an ideal
  $I = \ideal{f_1, \ldots, f_n} \subset S$ with
  $f_j \in S_{ d_j \cdot \alpha_0} = H^0(X,\OX(d_j \cdot
  \alpha_0))$, $d_j \in \N$ such that $V_X(I)$ is a complete intersection.
  Let $c$ be the codegree of $P$.  Then, for every $t < c$,
  $( \sum_j d_j - t) \; \alpha_0 \in \Reg(I)$ and
  $\left((\sum_j d_j - c) \; \alpha_0 , \alpha_0 \right)$ is a regularity pair.  
\end{theorem}

\begin{proof}
  As $P$ is a lattice polytope, $\alpha_0 \in \Picbpf(X)$, and, by
  \cite[Prop.~6.1.1]{cox_toric_2011}, for every $k \geq 0$,
  $k \, \alpha_0 \in \Picbpf(X)$. Fix an integer $t < c$.
  By \Cref{thm:vanishCohomExactGlobal}, we only need to prove that
  $H^p(X,\OX((\sum_{j \in \A} d_j - t) \; \alpha_0)) = 0$ for $p > 0$ and
  $\A \subset \{1,\dots,n\}$. If $\sum_{j \in \A} d_j \geq t$,
  these cohomologies vanish by Demazure vanishing. If $\sum_{j \in \A} d_j < t$, by Batyrev-Borisov vanishing,
  $H^p(X,\OX((\sum_{j \in \A} d_j - t) \; \alpha_0)) = 0$, for all
  $p \neq \dim(P)$. If $p = \dim(P)$, then the dimension of the $p$-th
  sheaf cohomology agrees with the number of lattice points in the
  interior of $(t - \sum_{j \in \A} d_j) \cdot P$, that is,
  $\dim_\C(H^p(X,\OX((\sum_{j \in \A} d_j - t) \, \alpha_0))) =
  \#(\relInt((t - \sum_{j \in \A} d_j) \cdot P) \cap M)$. As $t$ is
  strictly smaller than the codegree of the polytope $P$, for every
  $\A$, $(t - \sum_{j \in \A} d_j) \cdot P$ has no interior points, so
  $H^p(X,\OX((\sum_{j \in \A} d_j - t) \alpha_0))$ vanishes.
\end{proof}

Note that if $P$ is a hollow polytope, \Cref{thm:unmixed} improves the
pairs obtained in \Cref{thm:regCartier}. 
As a special case of \Cref{thm:unmixed}, we recover the Macaulay
bound \cite{lazard_grobner-bases_1983}, which bounds the
Castelnuovo-Mumford regularity for (classical) homogeneous square systems over $\mathbb{P}^n$. 

\begin{corollary}[Macaulay bound]
  \label{thm:macBound}
  Consider the smooth toric variety $\PP^n$ and let
  $S = \C[x_0,x_1,\dots,x_n]$ be its Cox ring. The class group of this
  variety is isomorphic to $\Z$ and $S$ is the standard $\Z$-graded
  polynomial ring. Consider the ideal $I = \ideal{f_1,\dots,f_n}$
  generated by homogeneous polynomials such that $\deg(f_i) = d_i$. If
  $V_X(I)$ is zero dimensional, then $(\sum_i d_i - n,1)$ is a regularity
  pair for $I$.
\end{corollary}

\begin{proof}
  The class group of $\PP^n$ is generated by the class of a basepoint
  free divisor $D$ representing a hyperplane
  $\PP^{n-1} \subset \PP^n$. The polytope associated to $D$ is the
  standard $n$-simplex $\Delta \in \R^n$, see \cite[Ex.~5.4.2]{cox_toric_2011}. Hence, we can think of each $f_i$ as a
  global section in $H^0(\PP^n,\OO_{\PP^n}(d_i \, D))$. As the codegree of
  $\Delta$ is $(n+1)$, Theorem \ref{thm:unmixed} tells us that
  $\left((\sum_i d_i - n) \, [D], [D]\right)$ is a
  regularity pair.
\end{proof}

\Cref{thm:unmixed} can be extended to non-lattice polytopes related to
nef $\Q$-Cartier divisors. We illustrate such an extension in the
important case of polynomial systems defined over a weighted
projective space. This is a non-smooth simplicial toric
variety \cite[Sec.~2.0]{cox_toric_2011}.
In what follows, we consider the weighted projective space
$X = \PP(q_0,q_1,\dots,q_n)$, where
$\gcd(q_0,\dots,\hat{q}_i,\dots,q_n) = 1$, for every $i$, and we fix
$\ell := \lcm(q_0,\dots,q_n)$. By \cite[Ex.~4.2.11]{cox_toric_2011},
the Class group of $X$ is generated by a nef
$\alpha_0 \in \Q\Picbpf(X)$ whose associated polytope is
$P := \{(a_0,\dots,a_n) \in \R^{n+1} ~|~ \sum_i q_i \, a_i = 1, q_i
\geq 0 \}$ and its Picard group is generated by
$\ell \, \alpha_0 \in \Picbpf(X)$.
\begin{theorem}
  \label{thm:weighted}
  Let $X = \PP(q_0,q_1,\dots,q_n)$ and
  fix $d_1,\dots,d_n \in \N$. Consider $f_1,\dots,f_n \in S$ such that
  $f_i \in S_{\ell \, d_i \, \alpha_0}$.
  Consider $d_{reg} > \ell \, ( \sum_{i=1}^n d_i )- \sum_{j=0}^n q_j$ such that
  $d_{reg}\,\alpha_0$ is $V_X(I)$-basepoint free.
  If $I := \ideal{f_1,\dots,f_n}$ is a complete intersection, then
  for any $V_X(I)$-basepoint free $d_0 \, \alpha_0 \in \Cl(X)_+$ such
  that $d_0 > 0$, we have that
  $(d_{reg} \, \alpha_0, d_0 \, \alpha_0)$ is a regularity pair for
  $I$.
\end{theorem}

\begin{proof}
  We consider the polytope $P$ associated to
  $\alpha_0 \in \Cl_+(X)$. Its codegree is $c = \sum_i q_i$ as this is the
  first dilation of $P$ for which $(1,\dots,1) \in \relInt(c \cdot P)$. Using Proposition \ref{thm:demBatVanishing} as in the proof of \Cref{thm:unmixed}, the theorem follows.
\end{proof}
A natural generalization of the unmixed case is the case where the
Newton polytopes of our polynomials are products of simpler polytopes.
We can extend our approach to this case using the well-known K\"unneth
formula.

\begin{proposition}[{K\"unneth formula}] \label{prop:kunneth}
  Let $P \in \R^{n},Q \in \R^{m}$ be two full-dimensional polytopes
  and consider their Cartesian product
  $P \times Q \subset \R^{n + m}$.
  Then, $X_{P} \times X_{Q} \simeq X_{P+Q}$
  \cite[Prop.~2.4.9]{cox_toric_2011} and
  $\Cl(X_P) \oplus \Cl(X_Q) \simeq \Cl(X_{P \times Q})$
  \cite[Ex.~4.1.2]{cox_toric_2011}. Moreover, given 
  $\alpha + \beta \in \Cl(X_{P \times Q})$, such that
  $\alpha \in \Cl(X_{P})$ and $\beta \in \Cl(X_{Q})$, we can write the
  sheaf cohomologies of the coherent sheaf
  $\OO_{X_{P \times Q}}\big( \alpha + \beta \big)$ in terms of the
  cohomologies of $\OO_{X_P}(\alpha)$ and $\OO_{X_Q}(\beta)$; for
  every $r$ we have
  $$
  H^r\big(X_{P \times Q},\OO_{X_{P \times Q}}\big( \alpha + \beta \big)\big)
  \simeq
  \bigoplus_{i + j = r}
  H^i(X_{P},\OO_{X_{P}}(\alpha))
  \otimes
    H^j(X_{Q},\OO_{X_{Q}}(\beta)).
  $$
\end{proposition}

\begin{proof}
  The statement follows immediately from the fact that toric varieties coming from
  polytopes are separated, see \cite[Thm.~3.1.5]{cox_toric_2011}, and \cite[Prop.~9.2.4]{kempf_algebraic_1993}.
\end{proof}

We illustrate our approach for toric varieties of this kind by
studying the regularity pairs of the multihomogeneous system
considered in \Cref{ex:27lines}.

\begin{example}[{Cont. of \Cref{ex:27lines}}]
  Let $\Delta \subset \R^2$ be the standard $2$-dimensional simplex in
  $\R^2$. The codegree of $\Delta$ is $3$. As we observed in the proof
  of \Cref{thm:macBound}, the toric variety $X_\Delta$ is
  $\PP^2$. Therefore, the toric variety
  $X = X_{\Delta} \times X_{\Delta}$ associated to the polytope
  $\Delta \times \Delta \subset \R^4$ is $\PP^2 \times \PP^2$.
  This variety is smooth $\Cl(X) = \Pic(X) = \Cl(X_\Delta) \times \Cl(X_\Delta)$ is
  generated by $(\beta,0)$ and $(0,\beta)$, where
  $\beta \in \Cl(X_\Delta)$ is the class of the divisor associated to
  $\Delta$ in $X_\Delta$. We use the natural identification $\Cl(X) \simeq \Z^2$ given by $(\beta,0) \sim (1,0)$ and $(0, \beta) \sim (0,1)$.
  As in \Cref{thm:macBound}, by \Cref{thm:demBatVanishing}, if
  $i \not\in \{0,2\}$, then
  $H^{i}(X_{\Delta},\OO_{\Delta}(a)) = 0$. Hence, we
  rewrite the previous cohomology using \Cref{prop:kunneth} as
  $$
  H^r(X,\OO_{X}(a, b)) = \bigoplus_{\substack{i,j \in \{0,2\} \\ i + j = r}}
  H^{i}(X_{\Delta},\OO_{X_\Delta}(a)) \otimes
  H^j(X_{\Delta},\OO_{X_\Delta}(b)).
  $$
  This shows that $H^1(X,\OX(a,b)) = H^3(X,\OX(a,b)) = 0$.
  Hence, we observe that the only possible non-zero higher sheaf cohomologies are,
  $$
  \begin{array}{r l}
  H^2(X,\OO_{X}(a, b)) = 
    &   H^{0}(X_{\Delta},\OO_{X_\Delta}(a)) \otimes
      H^2(X_{\Delta},\OO_{X_\Delta}(b)) ~ \oplus ~
    H^{2}(X_{\Delta},\OO_{X_\Delta}(a)) \otimes
  H^0(X_{\Delta},\OO_{X_\Delta}(b)) \\
  H^4(X,\OO_{X}(a, b)) = &
  H^{2}(X_{\Delta},\OO_{X_\Delta}(a)) \otimes
  H^2(X_{\Delta},\OO_{X_\Delta}(b))
  \end{array}
  $$
  Following the same argument as in \Cref{thm:macBound},
  $H^4(X,\OX(a,b)) \neq 0$ if and only if $a \leq -3$ and $b \leq -3$
  (blue area in \Cref{fig:vanishCohom}) and $H^2(X,\OX(a,b)) \neq 0$
  if and only if $a \leq -3$ and $b \geq 0$, or $a \geq 0$ and
  $b \leq -3$ (yellow and green areas in
  \Cref{fig:vanishCohom}). \Cref{fig:vanishCohom} shows the values
  $(a,b)$ for which some higher cohomology does not vanish,
  summarizing this analysis.
     \begin{figure}[h]
       \centering
       \subfloat[{Vanishing of cohomologies of
        $\OX(a,b)$}]{
        \label{fig:vanishCohom}
         {       \begin{tikzpicture}
         \draw[gray!50, thin, step=0.5] (-3.5,-3.5) grid (2,2);
         \draw[->] (-3.5,0) -- (2.2,0) node[right] {\small $a$};
         \draw[->] (0,-3.5) -- (0,2.2) node[above] {\small $b$};

         \foreach \x in {-6,...,3} {
           \draw (\x/2,0.1) -- (\x/2,-0.1) node[below] {\tiny \x};
         }
         \foreach \y in {-6,...,3} \draw (-0.1,\y/2) -- (0.1,\y/2) node[left= 5pt] {\tiny \y};




          \fill[yellow,opacity=0.2] (-3.5,2) rectangle (-1.5,0);

         \fill[green,opacity=0.2] (0,-1.5) rectangle (2,-3.5);

         \fill[blue,opacity=0.2] (-3.5,-3.5) rectangle (-1.5,-1.5);

       \end{tikzpicture}
 }
       }
       \qquad
       \subfloat[{Regularity of $I$ from \Cref{ex:27lines}}]{
         \label{fig:vanishCohomExample}
         {             \begin{tikzpicture}[scale=0.775]
         \draw[gray!50, thin, step=1] (0,0) grid (6.5,6.5);
         \draw[->] (-.2,0) -- (6.7,0) node[right] {\small $a$};
         \draw[->] (0,-.2) -- (0,6.7) node[above] {\small $b$};

         \foreach \x in {0,...,6} {
           \draw (\x,0.1) -- (\x,-0.1) node[below] {\tiny \x};
         }
         \foreach \y in {0,...,6} \draw (-0.1,\y) -- (0.1,\y) node[left= 5pt] {\tiny \y};

         \fill[yellow,opacity=0.2] (0,6.5) rectangle (3,3);

         \fill[green,opacity=0.2] (3,3) rectangle (6.5,0);

         \fill[blue,opacity=0.2] (0,0) rectangle (3,3);



       \end{tikzpicture}
         }
       }
       \caption[Regularity]{
         \Cref{fig:vanishCohom} shows the vanishing of the higher
         cohomologies of $\OX(a,b)$.
         \Cref{fig:vanishCohomExample} shows the regularity of $I$ from
           \Cref{ex:27lines}.  The white area corresponds to degrees
           $(a,b) \in \Reg(I)$. 
         The blue area
         corresponds to the degrees $(a,b)$ at which
         $H^4(X,\OX(a-6,b-6)) \neq 0$,
         the green area
         corresponds to
         $H^2(X,\OX(a-3,b-6)) \neq 0$ and the yellow area corresponds to
         $H^2(X,\OX(a-6,b-3)) \neq 0$.  }
  \end{figure}
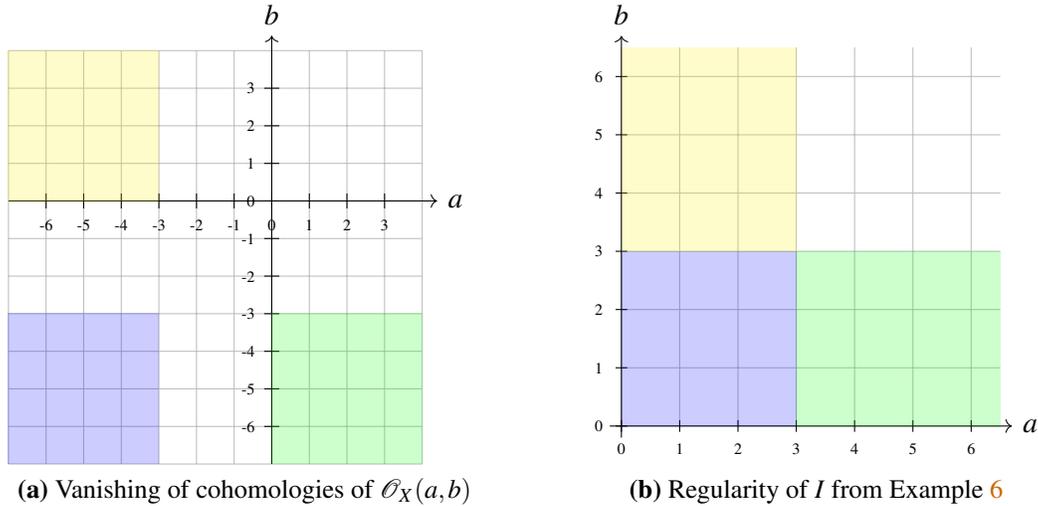

  Now we will apply \Cref{thm:vanishCohomExactGlobal}. If
  $a,b \in \N$, then the degree $(a, b) \in \Cl(X)$ is basepoint free,
  hence we only need to find the pairs $(a,b)$ such that the
  cohomologies in \eqref{cond:vanishHigherCohom} vanishes. In
  \Cref{fig:vanishCohomExample} we show the values for
  $(a,b) \in \N^2$ for which all these cohomologies vanish.
     In \Cref{ex:27lines}, we used the fact that by Theorem
     \ref{thm:regCartier}, $((6,6),(1,1))$ is a regularity pair.  Our
     previous analysis shows that $((4,4), (1,1))$ is a regularity
     pair as well.
     Using this new regularity pair reduces the size of the matrix of
     the map $\Res$ from \Cref{subsec:algorithm} from
     $1296 \times 2256$ to $441 \times 552$, which causes a speed-up by a factor $\approx 25$.
   \end{example}

   The analysis from the previous example can be generalized to recover
   the \emph{multihomogeneous Macaulay bound}, see
   \cite[Sec.~4]{bender_towards_2018}.

\begin{proposition}[Multihomogeneous Macaulay bound]
  \label{thm:multihomMacBound}
 Let $X := \PP^{n_1} \times \dots \times \PP^{n_s}$ and consider
  $n := n_1+\dots+n_s$ homogeneous polynomials $f_1,\dots,f_n$ such
  that $f_i \in S_{d_i}$, for each $i$, where
  $d_i \in \Cl_+(X) \simeq \N^s$ and $S = \bigoplus_{(a_1,\dots,a_s) \in \N^s} \bigotimes_{i = 1}^s
  \C[x_{i,0},\dots,x_{i,n_i}]_{a_i}$ is the Cox ring of $X$,
  \cite[Ex.~2.4.8]{cox_toric_2011}.
  If $V_X(\ideal{f_1,\dots,f_n})$ is a complete intersection, for any
  $b \in \Cl_+(X) \simeq \N^s$,
  $(\sum_{j=1}^n d_j - (n_1,\dots,n_s), b)$ is a regularity pair.
\end{proposition}

In the case of complete intersection over
$\PP^{n_1} \times \PP^{n_2}$, more can be said about the vanishing of
the higher cohomologies; see \cite{chardin_multigraded_2020}.

In this section, we showed how to improve the bounds from
\Cref{thm:regCartier} by combining \Cref{thm:vanishCohomExactGlobal}
with classical vanishing theorems. As mentioned before, in other cases not covered by Proposition \ref{thm:demBatVanishing}, the sheaf cohomologies can be computed explicitly \cite{altmann_immaculate_2018}. 

We emphasize that the results of Sections \ref{subsec:regpic} and \ref{subsec:generalizations} do not cover the \emph{overdetermined}/\emph{non-square} case, where the number of equations exceeds $n$. It remains an open problem to characterize regularity pairs for such systems. 

{
\footnotesize
\paragraph{Acknowledgments}
We are grateful to Marta Panizzut and Sascha Timme for bringing
\Cref{ex:27lines} to our attention.
We thank Laurent Bus\'e, Yairon Cid-Ruiz, Marco Ramponi, Joaqu\'in
Jacinto Rodriguez, Mesut {\c{S}}ahin, Ivan Soprunov, Pierre-Jean
Spaenlehauer, Josu\'e Tonelli Cueto and Elias Tsigaridas for answering
many questions regarding algebraic geometry, homological algebra and
useful references.
We thank Peter B\"urgisser, David Cox, Alicia Dickenstein and Teresa
Krick for their comments to improve a previous version of this paper.
Part of this work was done during the visit of the second author to TU
Berlin for the occasion of the MATH+ Thematic Einstein Semester on
Algebraic Geometry, Varieties, Polyhedra, Computation.
We thank the organizers of this nice semester for making this
collaboration possible.

The first author was funded by the ERC under the European’s Horizon
2020 research and innovation programme (grant agreement No 787840). The second author was supported by
the Research Council KU Leuven,
C1-project (Numerical Linear Algebra and Polynomial Computations),
and by
the Fund for Scientific Research--Flanders (Belgium),
G.0828.14N (Multivariate polynomial and rational interpolation and approximation),
and EOS Project no 30468160.
}

{
\footnotesize
\bibliographystyle{abbrv}
\setlength{\parskip}{0pt}
\bibliography{references}
}

\end{document}